%% file: revDuke-Schulte.tex
\documentclass[12pt]{amsart}
\def\color[#1]#2{}

\usepackage{graphicx} 
\usepackage{amsmath} 
\usepackage{amssymb} 
\usepackage{amsfonts} 
\usepackage{amsthm} 
\usepackage{longtable} 
\usepackage{comment}

\usepackage{color}

\usepackage[margin=1in]{geometry}

\newtheorem{theorem}{Theorem}

\newtheorem{lemma}{Lemma}

\newcommand{\po}{\mathcal{P}}

\newcommand{\Po}{\mathcal{P}}

\newcommand{\F}{\mathcal{F}}

\newcommand{\K}{\mathcal{K}}

\theoremstyle{comment}
\newtheorem*{acomment}{\color{BrickRed}{Comment}}

\definecolor{OliveGreen}{rgb}{.3,.5,.2}
\definecolor{MidnightBlue}{rgb}{.3,.4,.6}
\definecolor{BrickRed}{rgb}{.6,.1,.3}

\newcommand{\veci}[1]{\ensuremath{\mathbf{#1}}}

\usepackage{graphicx}

\parskip\medskipamount

\title[Cube-like Complexes]{Cube-like Incidence Complexes and Their Groups}

\author{Andrew C. Duke}
\address{Andrew C. Duke\\
Northeastern University,
Department of Mathematics,
Boston, MA 02115, USA}
\email{duke.a@husky.neu.edu}

\author{Egon Schulte}
\address{Egon Schulte\\
Northeastern University,
Department of Mathematics,
Boston, MA 02115, USA}
\email{schulte@neu.edu}

\date{\today}
\thanks{UDC Classification: 514 (Geometry)}
\thanks{E.S. supported by NSA-grant H98230-14-1-0124}

\begin{document}
\maketitle

\begin{abstract}
The article studies power complexes and generalized power complexes, and investigates the algebraic structure of their automorphism groups. The combinatorial incidence structures involved are cube-like, in the sense that they have many structural properties in common with higher-dimensional cubes and cubical tessellations on manifolds. Power complexes have repeatedly appeared in applications.
\end{abstract}

\section{Introduction}

Combinatorial, geometric or topological structures built from cubes or cube-like components have been studied extensively. The present paper studies two particularly interesting types of cube-like structures:\ power complexes and generalized power complexes. The former were first discovered by Danzer in the early 1980's (see \cite{dan,arp,esext}). Power complexes that are also abstract polytopes have attracted a lot of attention and have been described in detail in McMullen \& Schulte~\cite[Ch. 8]{arp}. They have repeatedly appeared somewhat unexpectedly in various mathematical applications, although often under a different name; for example, see Coxeter~\cite{crsp}, Effenberger \& K\"uhnel~\cite{eff}, K\"uhnel~\cite{kuet}, Beineke \& Harary~\cite{beh}, Ringel~\cite{ri}, Brehm, K\"{u}hnel \& Schulte~\cite{bks}, Pisanski, Schulte \& Weiss~\cite{psw} and Monson \& Schulte~\cite{mosc}. Structures arising from power complexes have also been used to describe fault-tolerant communication networks (see Bhuyan \& Agrawal~\cite{ba}, Dally~\cite{dally}, Fu, Chen \& Duh~\cite{fcd} and Ghose \& Desai~\cite{gd}).

The purpose of this article is to investigate the structure of the automorphism group of the power complexes $n^\K$ and describe a generalization of the polytopes $\mathcal{L}^{\K,\mathcal{G}}$ introduced in \cite[Ch. 8B]{arp}. Our discussion is in terms of incidence complexes, a broad class of ranked incidence structures closely related to abstract polytopes, ranked partially ordered sets, and incidence geometries (see~\cite{kom1,kom2}). In Section~\ref{bano} we review key facts about regular incidence complexes and generating sets of their automorphism groups. Section~\ref{pow} establishes basic properties of general power complexes, and Section~\ref{grouppow} determines the wreath product structure of their automorphism groups. Then in Section~\ref{nktwist} we explain how a regular power complex can be recovered from a small flag-transitive subgroup of its automorphism group by exploiting twisting techniques akin to those of \cite[Ch. 8]{arp}. Finally, Section~\ref{twist} describes generalized power complexes.

This paper is dedicated to our friend and colleague Nikolai Dolbilin, whose works include remarkable contributions to the study of cubical structures (for example, see Dolbilin, Shtanko \& Shtogrin~\cite{dss1,dss2}). 

\section{Basic Notions}
\label{bano}

Incidence complex are patterned after convex polytopes. The notion is originally due to Danzer~\cite{dan,kom1} and was inspired by Gr\"unbaum~\cite{grgcd}. The special case of abstract polytopes has recently attracted a lot of attention (see~\cite{arp}). Incidence complexes can also be viewed as incidence geometries or diagram geometries with a linear diagram (see Buekenhout-Pasini~\cite{bup}, Leemans~\cite{leem} and Tits~\cite{tib}), although here we study them from the somewhat different discrete geometric and combinatorial perspective of polytopes and ranked partially ordered sets.

Following~\cite{kom1} (and \cite{kom2}), an {\em incidence complex $\K$ of rank $k$\/}, or simply a {\em $k$-complex\/}, is a partially ordered set, with elements called {\em faces\/}, which has the following properties (I1),\ldots,(I4). \smallskip

\noindent
\textbf{(I1)}\  $\K$ has a least face $F_{-1}$ and a greatest face $F_k$, called the {\it improper\/} faces.
\smallskip

\noindent
\textbf{(I2)}\  Every totally ordered subset, or {\em chain\/}, of $\K$ is contained in a (maximal) totally ordered subset of $\K$ with exactly $k+2$ elements, called a {\em flag\/}.
\smallskip

These two conditions make $\K$ into a ranked partially ordered set, with a strictly monotone rank function with range $\{-1,0,\ldots,k\}$. A face of rank $i$ is called an $i$-\textit{face}.  A face of rank $0$, $1$ or $n-1$ is also called a {\em vertex\/}, an {\em edge\/} or a {\em facet}, respectively. 
\smallskip

\noindent
\textbf{(I3)}\  $\K$ is {\em strongly flag-connected\/}, meaning that if $\Phi$ and $\Psi$ are two flags of $\K$, then there is a finite sequence of flags $\Phi=\Phi_0,\Phi_1,\ldots,\Phi_{m-1},\Phi_{m}=\Psi$, all containing $\Phi\cap\Psi$, such that successive flags are {\em adjacent\/} (differ in just one face).
\smallskip

\noindent
\textbf{(I4)}\  There exist cardinal numbers $c_0,\ldots,c_{k-1}\geq 2$, for our purposes taken to be finite, such that, whenever $F$ is an $(i-1)$-face and $G$ a $(i+1)$-face with $F < G$, the number of $i$-faces~$H$ with $F<H<G$ is given by $c_i$.

Two adjacent flags are said to be {\em $i$-adjacent\/}, for $i=0,\ldots,k-1$, if they differ exactly in their $i$-faces. Condition (I4) can be rephrased by saying that the number of flags $i$-adjacent to a given flag equals $c_{i}-1$ for each $i$.

For an $i$-face $F$ and a $j$-face $G$ with $F < G$ we call
\[ G/F := \{ H \in \K \, | \, F \leq H \leq G \} \]
a \textit{section} of $\K$. This is an incidence complex in its own right, of rank $j-i-1$ and with cardinal numbers $c_{i+1},\ldots,c_{j-1}$. Usually we identify a $j$-face $G$ of $\K$ with the $j$-complex $G/F_{-1}$. Likewise, if $F$ is an $i$-face, the $(k-i-1)$-complex $F_k/F$ is called the {\em co-face\/} of $F$ in $\K$, or the \textit{vertex-figure} at $F$ if $F$ is a vertex.

An {\em abstract $k$-polytope\/}, or briefly {\em $k$-polytope\/}, is an incidence complex of rank $k$ such that $c_i=2$ for $i=0,\ldots,k-1$ (see~\cite{arp}). Note that a polytope is a complex in which every flag has precisely one $i$-adjacent flag for each $i$. An incidence complex of rank $1$ with $v$ vertices is also called a $v$-{\em edge} and is denoted $\{\}_v$ (here $v=c_0$).

The \textit{automorphism group} $\Gamma(\K)$ of an incidence complex $\K$ is the group of all order-preserving bijections of $\K$. A complex $\K$ is said to be \textit{regular} if $\Gamma(\K)$ is transitive on the flags of $\K$. Unlike regular polytopes, an arbitrary regular complex $\K$ need not be simply flag-transitive, so in general $\Gamma(\K)$ has nontrivial flag-stabilizers. 

Now let $\Gamma$ be any flag-transitive subgroup of the full automorphism group $\Gamma(\K)$ of a regular $k$-complex $\K$. It was shown in \cite{kom2} that any such group (including $\Gamma(\K)$ itself) has a distinguished system of generating subgroups obtained as follows (the proof for polytopes is given in \cite[Ch. 2B]{arp}). Suppose $\Phi:=\{F_{-1},F_0,\ldots,F_k\}$ is a fixed, or {\em base flag\/}, of $\K$, where~$F_i$ designates the $i$-face in $\Phi$ for each $i$. For each $\Omega\subseteq \Phi$ let $\Gamma_\Omega$ denote the stabilizer of $\Omega$ in $\Gamma$, so in particular $\Gamma_\emptyset = \Gamma$ and $\Gamma_\Phi$ is the stabilizer of $\Phi$. Then
\begin{equation}
\label{genga}
\Gamma = \langle R_{-1},R_0,\ldots,R_{k} \rangle ,
\end{equation}
where for $i=-1,0,\ldots,k$ the $i$\textsuperscript{th} subgroup $R_i$ is given by 
\[ R_{i} :=  \Gamma_{\Phi\setminus\{F_i\}} = \langle \varphi\in \Gamma \mid F_j\varphi =F_j \mbox{ for all } j\neq i\rangle .\]
Each $R_i$ contains $\Gamma_\Phi$, and even coincides with $\Gamma_\Phi$ when $i=-1$ or $k$. This shows that in generating $\Gamma$, the subgroups $R_{-1}$ and $R_{k}$ are redundant when $k>0$. In addition, 
\begin{equation}
\label{ci}
c_i := |R_{i}:\Gamma_\Phi | \quad\; (i=0,\ldots,k-1). 
\end{equation}
Moreover, the following commutation property holds at the level of groups (but not generally at the level of elements):
\begin{equation}
\label{commu}
R_i \cdot R_j = R_j \cdot R_i \qquad (-1\leq i < j-1 \leq k-1). 
\end{equation}

The stabilizers of subchains of the base flag in $\Gamma$ can be expressed in terms of the generating subgroups. More explicitly, if $\Omega\subseteq \Phi$ then 
\[ \Gamma_\Omega = \langle R_i \mid -1\leq i\leq k,\,F_{i}\not\in\Omega\rangle .\]

Further, for $I\subseteq\{-1,0,\ldots,k\}$ define $\Gamma_{I} := \langle R_i \mid i\in I\rangle$; when $I=\emptyset$ the latter is taken to mean $R_{-1}=\Gamma_\Phi$. Thus 
\[ \Gamma_{I} = \Gamma_{\{F_j\mid j\not\in I\}} \quad (I\subseteq\{-1,0,\ldots,k\}); \] 
or equivalently, 
\[ \Gamma_\Omega = \Gamma_{\{i\mid F_i \not\in\Omega\}} \quad (\Omega\subseteq \Phi). \]

Moreover, the group $\Gamma$ satisfies the following important {\em intersection property\/}: 
\begin{equation}
\label{intprop}
\Gamma_I \cap \Gamma_J = \Gamma_{I\cap J}\qquad (I,J\subseteq \{-1,0,\ldots,k\}) . 
\end{equation}

The partial order on $\K$ can be completely characterized in terms of the distinguished generating subgroups of $\Gamma$. In fact, 
\[ F_{i}\varphi \leq F_{j}\psi\; \longleftrightarrow\; \psi^{-1}\varphi \in \Gamma_{\{i+1,\ldots,k\}}\Gamma_{\{-1,0,\ldots,j-1\}}
\;\quad (-1\leq i\leq j\leq k;\, \varphi,\psi\in\Gamma) ,\]
or equivalently,
\begin{equation}
\label{partorder}
F_{i}\varphi \leq F_{j}\psi\, \longleftrightarrow \,
\Gamma_{\{-1,0,\ldots,k\}\setminus\{i\}}\varphi \cap \Gamma_{\{-1,0,\ldots,k\}\setminus\{j\}}\psi \neq \emptyset
\quad (-1\leq i\leq j\leq k;\, \varphi,\psi\in\Gamma). 
\end{equation}

It was also established in \cite{kom2} that if $\Gamma$ is any group with a system of generating subgroups $R_{-1},R_0,\ldots,R_k$ such that (\ref{commu}), (\ref{genga}) and (\ref{intprop}) hold, and $R_{-1}=R_k$, then $\Gamma$ is a flag-transitive subgroup of the automorphism group $\Gamma(\K)$ of a regular incidence complex $\K$ of rank $k$ (see \cite[Ch. 2E]{arp} for the proof for polytopes). As $i$-faces for $\K$ we take the right cosets of $\Gamma_{\{-1,0,\ldots,k\}\setminus\{i\}}$ for each $i$, and then define the partial order by (\ref{partorder}). The homogeneity parameters $c_0,\ldots,c_{k-1}$ then are determined by (\ref{ci}).

Abstract regular polytopes have only one flag-transitive subgroup, namely the full automorphism group $\Gamma(\K)$. In this case the flag stabilizer $\Gamma_\Phi$ is trivial and each subgroup $R_i$ (with $i=0,\ldots,k-1$) is generated by an involutory automorphism $\rho_i$ that maps $\Phi$ to its unique $i$-adjacent flag. The group of an abstract regular polytope is then what is called a {\em string C-groups\/} (see \cite{arp}), that is, the {\em distinguished involutory generators\/} $\rho_0, \ldots, \rho_{k-1}$ satisfy both the commutativity relations typical of a Coxeter group with string diagram, and the intersection property~(\ref{intprop}).

The regular complex polytopes in unitary complex space $\mathbb{C}^k$ are regular incidence complexes. Only the real polytopes among them are also abstract regular polytopes (see \cite{cox1,shep}). The  geometric (unitary) symmetry group of a regular complex polytope is a complex unitary reflection group in $\mathbb{C}^k$ acting flag-transitively (see \cite{sheptodd}); however, the combinatorial automorphism group is generally larger.
The regular complex $k$-cube 
\[ \gamma_{k}^{n} := n \{4\} 2 \{3\} 2 \cdots 2 \{3\} 2 ,\]
with $k\geq 1$ and $n\geq 2$ (and with $k-2$ entries $3$ in the symbol), has $n^k$ vertices, each with an ordinary $(k-1)$-simplex $\{3^{k-2}\}$ as vertex-figure, and $nk$ facets, each isomorphic to $\gamma_{k-1}^{n}$ if $k\geq 2$ (see \cite{cox1}). When $n=2$ this is the ordinary $k$-cube $\gamma_k = \gamma_{k}^2$. 

\section{Power complexes}
\label{pow}

The first published account of Danzer's power complex construction can be found in~\cite{esext}. (Danzer's  article about the construction announced in~\cite{dan} was never written.) The power complex $\mathcal{P}=n^\K$, with $n\geq 2$ and $\K$ a finite $k$-complex, is a $(k+1)$-complex with $n$ vertices on each edge and with each vertex-figure isomorphic to $\K$ (see \cite{esext}). The complexes $\Po$, with $n=2$ and $\K$ a polytope, are abstract polytopes and have been studied extensively (see \cite[Ch. 8C, D]{arp}). Power complexes are generalized cubes; when $\K$ has simplex facets, $\Po$ can also be viewed as a cubical complex (see \cite{buch,nov}). 

Let $n\geq 2$ and set $N:=\{1,\ldots,n\}$. Let $\K$ be a finite $k$-complex ($k>0$), not necessarily regular, with $v$ vertices and vertex-set $V:=\{1,\ldots,v\}$. Suppose $\K$ is {\it vertex-describable\/}, meaning that the faces of $\K$ are uniquely determined by their vertex-sets. Then $\Po$ will be a finite $(k+1)$-complex with vertex-set 
\begin{equation}
\label{nv}
N^v = \bigotimes_{i=1}^{v} N ,
\end{equation}
the cartesian product of $v$ copies of $N$; its $n^v$ vertices are written as row vectors 
$\varepsilon := (\varepsilon_1,\ldots,\varepsilon_v)$. Bearing in mind that the faces of $\K$ can be viewed as subset of $V$, we take as $j$-faces ($j\geq 0$) of $\Po$, for any $(j-1)$-face $F$ of $\K$ and any vector $\varepsilon = (\varepsilon_1,\ldots,\varepsilon_v)$ in $N^v$, the subsets $F(\varepsilon)$ of $N^v$ defined by
\begin{equation}
\label{fep}
F(\varepsilon) := \{(\eta_1,\ldots,\eta_{v})\in N^v \! \mid \eta_{i} = \varepsilon_{i} \mbox{ if } i\not\in F\}
\end{equation}
or, with slight abuse of notation, the cartesian product
\[ F(\varepsilon) \,:=\, ( \bigotimes_{i \in F} N ) \times ( \bigotimes_{i \not\in F} \{\varepsilon_i\} ). \]
Thus the $j$-face $F(\varepsilon)$ of $\mathcal{P}$ consists of the vectors in $N^v$ that coincide with $\varepsilon$ in precisely the components determined by the vertices of $\K$ not lying in $F$. It follows that, if $F$, $F'$ are faces of $\K$ and $\varepsilon=(\varepsilon_1,\ldots,\varepsilon_v)$, $\varepsilon' =(\varepsilon_{1}',\ldots,\varepsilon_{v}')$ belong to~$N^v$, then $F(\varepsilon) \subseteq F'(\varepsilon')$ in $\mathcal{P}$ if and only if $F \leq F'$ in $\K$ and $\varepsilon_i = \varepsilon_{i}'$ for each $i\not\in F'$. 

The set of all faces $F(\varepsilon)$, with $F$ a face of $\K$ and $\varepsilon$ in $N^v$, partially ordered by inclusion (and supplemented by the empty set as least face), is the desired {\em power complex\/}, $n^{\K}=\Po$. Each vertex-figure of $\Po$ is isomorphic to $\K$. If $F$ is a $(j-1)$-face of $\K$ and $\F:=F/F_{-1}$ is the $(j-1)$-complex determined by~$F$, then the $j$-faces of $\Po$ of the form $F(\varepsilon)$ with $\varepsilon$ in $N^v$ are isomorphic to the power complex $n^\F$ of rank~$j$. Moreover, $\po$ is regular if $\K$ is regular.  Complete proofs of these facts can be found in \cite{du,dusc} (see also \cite[Section 8D]{arp} and \cite{mosc,esext}).

The regular complex cubes $\gamma_{k}^n$ described in the previous section are examples of power complexes. In particular, $\gamma_{k}^{n} = n^{\alpha_{k-1}}$, where $\alpha_{k-1}$ denotes the $(k-1)$-simplex $\{3^{k-2}\}$. More examples are described below.

Power complexes behave nicely with respect to skeletons. Let $0\leq j\leq k-1$. Recall that the {\em $j$-skeleton\/} $skel_j(\K)$ of a $k$-complex $\K$ is the $(j+1)$-complex with faces those of $\K$ of rank at most $j$ or of rank $k$ (the $k$-face of $\K$ becomes the $(j+1)$-face of $skel_j(\K)$). If $\K$ is vertex-describable, then   
\begin{equation}
\label{skel}
skel_{j+1}(n^\K) = n^{skel_j(\K)} 
\end{equation}
(see \cite{dusc}). For example, if $\K$ is a $v$-edge $\{\}_v$ (of rank $1$), then identifying $\K$ with $skel_{0}(\alpha_{v-1})$ gives
\[ n^{\{\}_v} = n^{skel_{0}(\alpha_{v-1})} = skel_{1}(n^{\alpha_{v-1}}) = skel_{1}(\gamma_v^n).\]
Thus the $2$-complex $n^{\{\}_v}$ is isomorphic to the $1$-skeleton of the unitary complex $v$-cube $\gamma_v^n$ described above. 

\section{The group of a power complex}
\label{grouppow}

In this section we determine the structure of the automorphism group of a power complex. Let again $\Po:=n^\K$, where $n$ and $\K$ are as above. The automorphism group $\Gamma(\Po)$ of~$\Po$ contains a subgroup isomorphic to $S_n\wr \Gamma(\K) = S_{n}^{v}\rtimes \Gamma(\K)$, the wreath product of $S_n$ and $\Gamma(\K)$ defined by the natural action of $\Gamma(\K)$ on the vertex-set of $\K$ (see \cite{dusc}). In particular, an automorphism $\varphi$ of $\K$ determines an automorphism $\widehat{\varphi}$ of $\Po$ as follows. For a vertex $\varepsilon = (\varepsilon_1, \dotsc, \varepsilon_v)$ of $\Po$, we have
\begin{equation}
\label{autk}
\varepsilon\widehat{\varphi} = (\varepsilon_{1\varphi},\dotsc, \varepsilon_{v\varphi}) =: \varepsilon_{\varphi},
\end{equation}
and, more generally, for a face $F(\varepsilon)$ of $\Po$, 
\begin{equation}
\label{autfacek} 
F(\varepsilon)\widehat{\varphi} := (F\varphi)(\varepsilon_\varphi). 
\end{equation}
More generally, for $\sigma \in S_n^v$ and $\varphi \in \Gamma(\K)$, the typical element $\theta = \sigma\varphi$ of $S_n^v\ltimes \Gamma(\K)$ acts on $\Po$ according to
\begin{equation}
\label{action}
F(\varepsilon)\theta := F(\varepsilon\sigma)\widehat{\varphi} = (F\varphi)((\varepsilon\sigma)_\varphi).
\end{equation}
Note that $S_n\wr \Gamma(\K)$ is vertex-transitive on $\Po$ and that its vertex stabilizers in $S_n\wr \Gamma(\K)$ are isomorphic to $S_{n-1}\wr \Gamma(\K)$.
\medskip

\begin{theorem}
\label{grstruc} 
Let $\mathcal{K}$ be a finite vertex-describable incidence complex with automorphism group $\Gamma(\K)$, and let $n\geq 2$. Then the power complex $n^\K$ has automorphism group isomorphic to~$S_n \wr \Gamma(\K)$. 
\end{theorem}

\begin{proof}
We already mentioned that $S_n \wr \Gamma(\K)$ is a vertex-transitive subgroup of $\Gamma(\Po)$. In particular, $\Gamma(\Po)$ contains a vertex-transitive subgroup isomorphic to $S_n^v$; for each $i$ in $V$, a copy of $S_n$ acts on the $i$\textsuperscript{th} component of the vectors in $N^v$ leaving all other components unchanged. Additionally, by (\ref{autk}) applied with $\eta = \veci{1}:=(1,\dotsc, 1)$, the group $\Gamma(\K)$ is naturally embedded in $\Gamma(\Po)$ as a subgroup of the vertex stabilizer of the {\em base} vertex $(1,\dotsc, 1)$ of $\Po$.

We will now show that $S_n\wr \Gamma(\K)$ coincides with $\Gamma(\Po)$. First note that by the vertex-transitivity of $S_n\wr \Gamma(\K)$ on $\Po$ it is sufficient to investigate the structure of the stabilizer of a {\em base\/} vertex of $\Po$ in $\Gamma(\Po)$. In fact, up to an element from $S_n\wr \Gamma(\K)$ each automorphism of $\Po$ can be made to coincide with an automorphism in the stabilizer of the base vertex in $\Gamma(\Po)$.

As base vertex we choose $\veci{1}$. For $i \in V$ and $j \in N$, let 
\[ \veci{1}_i^j := (1, \dotsc, 1, j, 1, \dotsc, 1)\] 
be the vector obtained from $\veci{1}$ by replacing the $i$\textsuperscript{th} component by $j$. In designating edges of~$\Po$, we find it convenient to use two separate notations for vertices of $\K$, namely $1, \dotsc, v$ and $G_1, \dotsc, G_v$, with the understanding that $G_i$ is identified with $i$ for $i = 1, \dotsc, v$. The first notation views a vertex of $\K$ as an element of $V$, and the second emphasizes that a vertex is a face of $\K$. Notice that $G_i(\veci{1})$ is an edge of $\Po$ with vertices $\veci{1}_i^j$ for $j\in N$.

Now suppose an element $\rho \in \Gamma(\Po)$ fixes the base vertex $\veci{1}$ of $\Po$, that is, $\veci{1}\rho = \veci{1}$. We will show that there are elements $\varphi \in \Gamma(\K)$ and $\sigma \in S_n^v$ such that $\rho\varphi\sigma$ is the identity automorphism of $\Po$. We begin by constructing $\varphi$.

First note that, since $\rho$ fixes $\veci{1}$, it induces an automorphism on the vertex-figure $F_{d}(\veci{1})/\veci{1}$ of $\Po$ at $\veci{1}$, which is isomorphic to $\K$. If $\varphi$ is the automorphism of $\K$ corresponding to the restriction of $\rho^{-1}$ to the vertex-figure $F_{d}(\veci{1})/\veci{1}$, then $\rho' = \rho\varphi$ is an automorphism of $\Po$ that acts trivially on the entire vertex-figure $F_{d}(\veci{1})/\veci{1}$ at $\veci{1}$. Thus $F(\veci{1})\rho' = F(\veci{1})$ for all faces $F$ of $\K$. Furthermore, since in particular each edge $G_i(\veci{1})$ is fixed under $\rho'$, the vertices $\veci{1}_i^j$ of $G_i(\veci{1})$ are permuted among each other by $\rho'$.

Next we construct $\sigma$ by viewing $\rho'$ as a permutation on $N^v$. Recall that $S_n^v$ is a subgroup of $\Gamma(\Po)$. More explicitly define  
\[S_n^{(i)}  := \langle 1\rangle\! \times\! \dotsm\! \times\! \langle 1 \rangle \!
\times\! S_n \!\times  \!\langle 1\rangle\!\times\! \dotsm\! \times\!\langle 1\rangle,\] 
where the $S_n$ occurs in position $i$. Thus $S_n^v = S_n^{(1)} \times \dotsm \times S_n^{(v)}$. For each $i \in V$ there exists an automorphism $\sigma_i$ of $\Po$ in $S_n^{(i)}$ such that 
$\veci{1}^j_i\rho'\sigma_i= \veci{1}^j_i$ 
for all $j$, and 
$\veci{1}^j_k\rho'\sigma_i = \veci{1}^j_k\rho'$ 
for all $k\neq i$ and all $j$. Setting 
$\sigma := \sigma_1\sigma_2 \dotsc \sigma_v$ 
and taking 
$\rho'' := \rho'\sigma = \rho\varphi\sigma$, 
we observe that $\rho''$ fixes the vertices and edges in Figure~\ref{fig1}, that is, all vertices $\veci{1}_i^j$ with $i \in V$ and $j \in N$, and all edges $G_i(\veci{1})$ with $i \in V$.

\begin{figure}[htb]
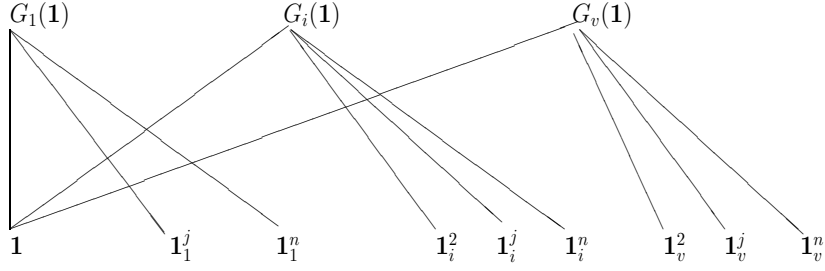

\centering
\resizebox{300pt}{100pt}{\input fig24.latex}
\caption{Vertices and edges fixed by $\rho''$.}
\label{fig1}
\end{figure}

The {\em weight\/} of a vertex $\varepsilon = (\varepsilon_1, \dotsc, \varepsilon_v)$ is the number of components $\varepsilon_i$ distinct from~$1$. The set $S:=\{i\mid \varepsilon_{i}\neq 1\}$ is called the {\em support\/} of $\varepsilon$. Similarly, define the {\em weight\/} of an edge $G_j(\varepsilon)$ as the number of $\varepsilon_l$, $l\neq j$, that are distinct from~$1$. We have just established the fact that $\rho''$ fixes all vertices of weight $1$ and all edges of weight $0$, so in particular $\rho''$ fixes all vertices and edges of weight $0$.

We will show by induction on $m$ that all vertices and edges of weight $m$ are fixed by $\rho''$. The statement holds for $m=0$ as well as for vertices when $m=1$. Now let $m \geq 1$ and assume inductively that $\rho''$ fixes all vertices and edges of weight $m-1$.

In the proof for vertices we may additionally assume that $m\geq 2$, since the statement is true for $m=1$. Every vertex of weight $m$ lies in exactly $m$ edges of weight $m-1$; namely a vertex $\varepsilon$ of weight $m$ with support $S$ is incident to all edges of the form $G_{i}(\varepsilon)$ with $i\in S$. Let $\varepsilon$ be such a vertex and let $\varepsilon\rho'' = \eta$; we need to show that $\eta = \varepsilon$. Since $\rho''$ is known to fix all edges of weight $m-1$, we have 
\[ \eta = \varepsilon\rho''\in G_i(\varepsilon)\rho'' = G_i(\varepsilon) \] 
and hence $\eta_j = \varepsilon_j$ for $j\neq i$, for each $i\in S$. Hence we have $\eta=\varepsilon$ since $|S|=m \geq 2$, and $\rho''$ fixes $\varepsilon$ as desired. This proves the statement for each vertex of weight $m$.

To establish the statement for edges of weight $m\,(\geq 1)$, let $\varepsilon$ be a vertex of weight $m$ with support $S$. An edge $G_j(\varepsilon)$ has weight $m-1$ or $m$ according as $j\in S$ or $j\notin S$. By inductive hypothesis the edges $G_j(\varepsilon)$ with $j\in S$ are fixed by $\rho''$. Now let $j\notin S$ and let $G_j(\varepsilon)\rho'' = G_{j'}(\varepsilon)$ for some $j'$ where necessarily $j'\notin S$; note here that $\rho''$ fixes $\varepsilon$. We need to show $j' = j$. Let $i_m$ denote the largest index in $S$, so in particular $j \neq i_m$. Define the two vertices $\mu = (\mu_1,\ldots,\mu_v)$ and $\nu=(\nu_1,\ldots,\nu_v)$ by 
\begin{equation}
\label{eqmu}
\mu_{i}=1, 2,\varepsilon_{i} \mbox{ according as } i = i_m \mbox{ or }i=j \mbox{ or }i\neq i_{m},j,
\end{equation}
and
\begin{equation}
\label{eqnu}
\nu_{i}=2,\varepsilon_{i} \mbox{ according as } i=j \mbox{ or }i\neq j.
\end{equation}
Then the sequence
\begin{equation}
\label{eqflag1}
\varepsilon, G_j(\varepsilon), \nu,G_{i_m}(\nu),\mu
\end{equation}
consists of sequentially incident vertices and edges of $\Po$, and so does its image under $\rho''$. Now since  $G_j(\varepsilon)\rho'' = G_{j'}(\varepsilon)$ and the vertices $\varepsilon$ and $\mu$ of weight $m$ are fixed by $\rho''$, we obtain the image sequence
\begin{equation}
\label{eqflag2}
\varepsilon, G_{j'}(\varepsilon), \nu\rho'',G_{i_m}(\nu)\rho'',\mu.
\end{equation}
However, $\mu\in G_{i_m}(\nu)\rho''$ just means that $G_{i_m}(\nu)\rho'' = G_l(\mu)$ for some $l \in V$; and 
\[ \nu\rho'' \in G_{i_m}(\nu)\rho'' = G_l(\mu)\] 
implies that $(\nu\rho'')_i = \mu_i$ for $i\neq l$. (For a vector $\eta=(\eta_1,\ldots,\eta_{v})$ we let $(\eta)_i:=\eta_i$ for each~$i$.) On the other hand, from \eqref{eqflag2} we have $\nu\rho''\in G_{j'}(\varepsilon)$ and hence $(\nu\rho'')_i = \varepsilon_i$ for $i \neq j'$. It follows that, since $(\nu\rho'')_i = \mu_i$ for $i\neq l$ and $(\nu\rho'')_i = \varepsilon_i$ for $i\neq j'$, we must have $\mu_i = \varepsilon_i$ for $i\neq j', l$. By the definition of $\mu$ this forces $\{l, j'\} = \{i_m, j\}$. Since $j, j'\notin S$ but $i_m\in S$, this then gives $j = j'$ (and $l = i_m$). But this is what we want to show. Thus all the edges $G_j(\varepsilon)$, with $\varepsilon$ of weight $m$, are fixed by $\rho''$ as well. But this includes all edges of $\Po$ of weight $m$, so the induction is complete.

Therefore, $\rho''$ fixes all vertices and edges of $\Po$. Now, since $\Po$ is vertex-describable all faces of $\Po$ are uniquely determined by their vertex-sets. It follows that $\rho''$ must fix every face of~$\Po$. Therefore, $\rho''$ is the identity on $\Po$, and $\rho = \sigma^{-1}\varphi^{-1}\in S_n\wr \Gamma(\K)$.
\end{proof}

Based on the above structure result for the automorphism group we also use the alternate notation for $n^\K$ of $n\wr\K$. Then Theorem~\ref{grstruc} says that
\[ \Gamma(n\wr\K)=S_n \wr \Gamma(\K) .\]
When $\K$ is a $v$-edge $\{\}_v$ we also write $n\wr v$ in place of the $2$-complex $n^{\{\}_v}=n\wr {\{\}_v}$. In particular, 
\[ \Gamma(n\wr v) = S_{n}\wr S_{v} .\]
\smallskip

For the remainder of this section we assume that $\K$ is regular. Then the $(k+1)$-complex $\Po$ is also regular and the distinguished generating subgroups $\widehat{R}_{-1}, \widehat{R}_0, \dotsc, \widehat{R}_{k+1}$ for $\Gamma(\Po)=S_n \wr \Gamma(\K)$ can be obtained as follows (\cite{esnotes}).

Suppose $\Gamma(\K) = \langle R_{-1}, R_0, \dotsc, R_{k} \rangle$, where $R_{-1}, R_0, \dotsc, R_{k}$ are the distinguished generating subgroups for $\Gamma(\K)$ with respect to a flag $\Phi = \{F_{-1}, F_0, \dotsc, F_{k}\}$ of $\K$. Here we take $F_0 = 1$ in the labeling of the vertices of $\K$. Recall that $R_{-1} = R_{k}$ is the stabilizer of $\Phi$ in~$\Gamma(\K)$. As base vertex for $\Po$ we now take $\veci{n} = (n,\dotsc, n)$, not $\veci{1}$, and as base flag we choose
\[\Phi(\veci{n}) := \{\emptyset, \veci{n}, F_0(\veci{n}), \dotsc, F_{k}(\veci{n})\}.\]
Here $F_{k}(\veci{n}) = N^v$ and
\[F_0(\veci{n}) = N \times \{n\} \times \dotsm \times \{n\}.\]

First we determine the stabilizer of the base flag, $\widehat{R}_{-1}$. Clearly, the subgroup $S_{n-1}\wr R_{-1}$ of $\Gamma(\Po)$ stabilizes $\Phi(\veci{n})$, since $R_{-1}$ stabilizes $\Phi$ in $\K$ and the subgroup $S_{n-1}^v$ of $S_{n}^v$ (with each component subgroup $S_{n-1}^{(i)}$ acting on $1, \dotsc, n-1$) stabilizes~$\veci{n}$. We claim that $S_{n-1}\wr R_{-1}$ is already the full flag stabilizer of $\Phi(\veci{n})$. Suppose $\theta \in \Gamma(\Po)=S_n^v\ltimes \Gamma(\K)$ fixes $\Phi(\veci{n})$. If $\theta=\sigma\varphi$ with $\varphi \in \Gamma(\K)$ and $\sigma=\sigma_{1}\ldots\sigma_v \in S_n^v$ (with $\sigma_{i}\in S_{n}^{(i)}$ as above), then \eqref{action} gives 
\[ F_j(\veci{n}) = F_j(\veci{n})\theta = (F_j\varphi)((\veci{n}\sigma)_\varphi) \]
for each $j = -1, 0, \dotsc, k$. Hence, since $\K$ is vertex-describable, $F_j\varphi = F_j$ for all $j$ and so $\varphi$ lies in $R_{-1}$, the stabilizer of $\Phi$ in $\Gamma(\K)$. Moreover, 
\[ \veci{n}\sigma= \veci{n}\theta\varphi^{-1} = \veci{n}\varphi^{-1} = \veci{n}\] 
and hence $(n)\sigma_{i}=n$ for each $i=1,\ldots,v$, so $\sigma$ lies in the subgroup $S_{n-1}^v$ of $S_n^v$. Thus the stabilizer of $\Phi(\veci{n})$ is given by $S_{n-1} \wr R_{-1}$, that is, $\widehat{R}_{-1}=S_{n-1} \wr R_{-1}$. 

Next we find $\widehat{R}_{0}$. Suppose $\theta = \sigma\varphi$ fixes all faces in $\Phi(\veci{n})$ except possibly the vertex $\veci{n}$. Then $\varphi$ fixes $\Phi$ in $\K$ and thus $\varphi \in R_{-1}$. Moreover, since $F_0((\veci{n}\sigma)_\varphi) = F_0(\veci{n})$, the vertices $\veci{n}$ and 
$(\veci{n}\sigma)_\varphi = (n\sigma_{1\varphi},\ldots,n\sigma_{v\varphi})$ coincide in all but possibly the first components. Hence, since $F_{0}=1$ and therefore $1\varphi=1$, we must have $(n)\sigma_{j}=n$ for $j\geq 2$. 
It follows that $\sigma \in S_n \times S_{n-1}^{v-1}$ and 
\[ \theta \in \langle S_n \times S_{n-1}^{v-1}, R_{-1} \rangle = (S_n \times S_{n-1}^{v-1})\ltimes R_{-1}.\]
Conversely, the elements of this group fix all faces of the base flag except possibly the vertex. Thus 
$\widehat{R}_{0}= (S_n \times S_{n-1}^{v-1})\ltimes R_{-1}$.

For $\widehat{R}_{i}$ with $i=1\ldots,k$, suppose $\theta=\sigma\varphi$ fixes all faces of $\Phi(\veci{n})$ except the $i$-face $F_i(\veci{n})$. Then 
\[ F_j(\veci{n}) = F_j(\veci{n})\theta = (F_j\varphi)((\veci{n}\sigma)_\varphi) \;\;\; (j \neq i). \]
For $j = -1$ this shows that $\veci{n}\theta = \veci{n}$, that is, $(\veci{n}\sigma)_\varphi = \veci{n}$ and hence $\sigma \in S_{n-1}^v$. Also $F_j\varphi = F_j$ for $j \neq i$; so $\varphi \in R_i$, the stabilizer of $\Phi\setminus \{F_i\}$ in $\K$, and $\theta \in S_{n-1} \wr R_i$. Conversely, $S_{n-1} \wr R_i$ is a subgroup of $\widehat{R}_{i}$. Thus $\widehat{R}_{i} = S_{n-1} \wr R_i$.

In summary, we have the following result.
Therefore, we have the following as the generating subgroups for $\Gamma(\Po) =S_n \wr \Gamma(\K)$.

\begin{theorem}
If $\K$ is a finite vertex-describable regular $k$-complex with automorphism group $\Gamma(\K) = \langle R_{-1}, R_0, \dotsc, R_{k} \rangle$, then $\Po$ is a regular $(k+1)$-complex and the distinguished generating subgroups $\widehat{R}_{-1}, \widehat{R}_0, \dotsc, \widehat{R}_{k+1}$ of its automorphism group $\Gamma(\Po)=S_n \wr \Gamma(\K)$ are given by 
\begin{align*}
\widehat{R}_0     &\,=\, (S_n \times S_{n-1}^v)\ltimes R_{-1},\\
\widehat{R}_i &\,=\, S_{n-1} \wr R_{i-1}\text{ for } i = 1, \dotsc, k,\\
\widehat{R}_{k+1} &\,=\,  \widehat{R}_{-1} \,=\, S_{n-1} \wr R_{-1} .\\
\end{align*}
\end{theorem}

If $\K$ is a regular complex, the automorphism group of the power complex $n^\K$ has a large supply of flag-transitive subgroups. In fact, we have the following lemma.

\begin{lemma}
\label{transsub}
Let $\K$ be regular. If $U$ is a subgroup of $S_n$ acting transitively on $\{1,\dots ,n\}$ and $\Lambda$ is a flag-transitive subgroup of $\Gamma(\K)$, then the subgroup $U\wr\Lambda = U^v\ltimes \Lambda$ of $S_n\wr \Gamma(\K)$ is a flag-transitive subgroup of $\Gamma(n^\K)$. 
\end{lemma}

\begin{proof}
First we note that $U\wr \Lambda = U^v\rtimes \Lambda$, where $U^v$ is the direct product of $v$ copies of $U$ acting in the obvious way on the vertex set $N^v$ of $n^\K$. In particular, $U^v$ acts transitively on $N^v$ and hence is a vertex transitive subgroup of $\Gamma(n^\K)$. On the other hand, $\Lambda$ acts flag-transitively on the vertex-figure at the base vertex. Hence $U\wr \Lambda = \langle U^v, \Lambda\rangle$ acts flag-transitively on $n^\K$.
\end{proof}

The (unitary) geometric symmetry group of the regular complex $k$-cubes $\gamma_{k}^n$ is isomorphic to $C_{n}\wr S_k$ and is a simply flag-transitive subgroup of the full combinatorial automorphism group $S_{n}\wr S_k$ of $\gamma_{k}^n$, obtained from Lemma~\ref{transsub} when $U=C_n$.

There are many other interesting power complexes besides $\gamma_{k}^n$. Here we briefly describe three more examples, two related to complex or real polytopes, and another to projective planes. More examples are discussed in \cite{du}.

First, consider the two regular power complexes $n^\K$ obtained when $\K$ is the regular $4$-crosspolytope $\{3,3,4\}$ or the regular complex polygon ($2$-polytope) $3\{3\}3$. For $\{3,3,4\}$, the geometric symmetry group (denoted $[3,3,4]$) is isomorphic to its combinatorial automorphism group. For $3\{3\}3$, the geometric symmetry group (denoted $3[3]3$) is isomorphic to $\rm{SL}(2,3)$; this is a subgroup of index $2$ in the combinatorial automorphism group of $3\{3\}3$, which is isomorphic to $\rm{GL}(2,3)$. Both $\{3,3,4\}$ and $3\{3\}3$ have $8$ vertices, so \eqref{skel} with $j=0$ shows that the corresponding power complexes $n^{\{3,3,4\}}$ and $n^{3\{3\}3}$ have isomorphic 1-skeletons, each with $n^8$ vertices and $8\cdot n^7$ edges ($n$-edges). The vertex-figures are isomorphic to $\{3,3,4\}$ and $3\{3\}3$, respectively, and the facets to complex $4$-cubes $\gamma_{4}^{n} = n^{\{3,3\}}$ and $2$-complexes~$n\wr 3$. The automorphism groups of $n^{\{3,3,4\}}$ and $n^{3\{3\}3}$ are $S_n \wr [3,3,4] = S_n^8 \ltimes [3,3,4]$ of order $384(n!)^8$ and $S_n\wr GL(2,3) \cong S_n^8 \ltimes GL(2,3)$ of order $48(n!)^8$, respectively. 

When $n=2$ we know from \cite[Cor. 8C6]{arp} that $2^{\{3,3,4\}}$ is the cubical regular 5-toroid $\{4,3,3,4\}_{(4,0,0,0)}$ tessellating the $4$-torus. This is an abstract regular $5$-polytope with 256 vertices, each with a $4$-crosspolytope as vertex-figure, and with 256 facets, each a 4-cube $\{4,3,3\}$.
The $3$-complex $2^{3\{3\}3}$ has 256 vertices and $8\cdot 2^7$ edges, which are $2$-edges, and they form the $1$-skeleton of $\{4,3,3,4\}_{(4,0,0,0)}$. Each facet of $2^{3\{3\}3}$ is of type $2\wr 3$, the $1$-skeleton of the ordinary $3$-cube, and there are 256 such facets. Figure~\ref{pow333} shows the vertex-figure at $\veci{1} = (11111111)$ and the facet $124(\veci{1})$, where $124$ represents the $3$-edge of $3\{3\}3$ with vertices $1,2,4$.  The vertex-figures are isomorphic to $3\{3\}3$; in particular, each vertex of $2^{3\{3\}3}$ lies in eight edges ($2$-edges) and eight facets, each isomorphic to the $1$-skeleton of the $3$-cube.  
\begin{figure}[h!]
\centering
\resizebox{250pt}{170pt}{\input newfig28.latex}
\caption{Vertex-figure of $2^{\mathrm{3\{3\}3}}$ at $\veci{1}$.}
\label{pow333}
\end{figure}

Our final example is the regular $3$-complex $n^{\mathrm{PG}(2,2)}$, where $\mathrm{PG}(2,2)$ denotes the Fano plane (with $7$ points and $7$ lines), the projective plane over a field of order $2$. In the Fano plane, each pair of vertices is colinear with exactly one additional vertex. Around each vertex of $n^{\mathrm{PG}(2,2)}$ there are then seven facets. Therefore in $n^{\mathrm{PG}(2,2)}$ each pair of edges ($n$-edges) forms a facet with exactly one additional edge. The complex $n^{\mathrm{PG}(2,2)}$ has $n^7$ vertices, $7\cdot n^6$ edges, and $7\cdot n^4$ facets, each isomorphic to $n\wr 3$. The vertex-figures are isomorphic to $\mathrm{PG}(2,2)$, so each vertex lies in seven edges ($n$-edges) and seven facets, and three facets meet at each edge. The automorphism group of $n^{\mathrm{PG}(2,2)}$ is $S_n\wr PGL(3,2) \cong S_n^7 \ltimes PGL(3,2)$.

In the power complex $2^{\mathrm{PG}(2,2)}$ there are 128 vertices, 448 edges ($2$-edges) and 112 facets. Each facet is isomorphic to $2\wr 3$, the $1$-skeleton of the ordinary cube. Figure~\ref{powfano} shows the three facets that meet at an edge.
\begin{figure}[h!]
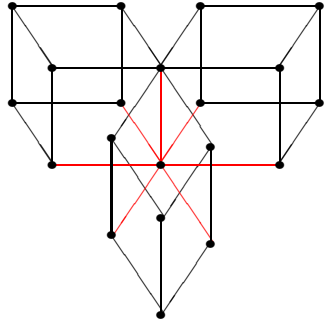

\centering
\resizebox{120pt}{120pt}{\input fig210.latex}
\caption{Three facets of $2^{\mathrm{PG}(2,2)}$ with a common edge. Each vertex lies in seven edges.}
\label{powfano}
\end{figure}

The power complexes $2^{\K}$ obtained from abstract polytopes $\K$ with simplex facets, themselves have cubical facets and often give cubical complexes with interesting topological properties; for example, see \cite{bks,kusc,kuet}. The cubical regular $5$-toroid $2^{\{3,3,4\}}$ described above is of this kind. If $\K$ is a regular or chiral torus map of type $\{3,6\}_{(p,q)}$, then the cubical complex is a 3-dimensional pseudomanifold in which the vertices are ``toroidal singularities" with a neighborhood given by a ``cone over a 2-torus". These complexes are quotients of the regular tessellation $\{4,3,6\}$ in hyperbolic 3-space.

\section{Recovering $n^{\mathcal{K}}$ from a small flag-transitive subgroup}
\label{nktwist}

Let $\mathcal{K}$ be a finite regular incidence complex of rank $k \geq 1$ with $v$ vertices, vertex set $V$, and base flag  $\{F_{-1}, F_0, F_1, \dotsc, F_k\}$. Let the automorphism group of $\mathcal{K}$ relative to the base flag be given by  $\Gamma(\K) = \langle R_0, R_1, \dotsc, R_{k-1}\rangle$ (since $k\geq 1$ we can suppress $R_{-1}$ and $R_k$). We begin by constructing a regular $(k+1)$-complex with vertex-figure isomorphic to $\mathcal{K}$, and then establish isomorphism with the power complex $n^{\mathcal{K}}$.

Consider the direct product $W$ of $v$ copies of the cyclic group $C_n$ of order $n$, each $C_n$ with a generator $\sigma_F$ indexed by a vertex $F$ of $\mathcal{K}$. Thus
\[W \,=\, \langle \sigma_F | F\in V\rangle  \,=\, \bigotimes_{F\in V}\langle \sigma_F \rangle \,=\, C_n^v.\]
We view $W$ as a subgroup of the semi-direct product $\Gamma:= W \ltimes \Gamma(\K) = C_n\wr \Gamma(\K)$, where the conjugation action of $\Gamma(\K)$ on $W$ is determined by
\[\varphi^{-1}\sigma_F\varphi = \sigma_{F\varphi}\quad (F\in V, \varphi \in \Gamma(\K)).\]

Define the subgroups $\widehat{R}_0, \widehat{R}_1, \dotsc, \widehat{R}_k$ of $\Gamma$ by
\begin{equation}
  \label{eq:twist-Ri}
\widehat{R}_i =
\begin{cases}
  \langle \sigma_{F_0} \rangle &\mbox{if } i = 0,\\
  R_{i-1} &\mbox{if }1 \leq i \leq k.
\end{cases}
\end{equation}
Since $\Gamma(\K)$ acts vertex-transitively on $\mathcal{K}$, each generator $\sigma_F$ of $W$ lies in $\langle \widehat{R}_0, \dotsc, \widehat{R}_k \rangle$; in fact, if $F\in V$ and $F = F_0\varphi$ for some $\varphi \in \Gamma(\K)$, then $\sigma_F = \varphi^{-1}\sigma_{F_0}\varphi$. It follows that
\begin{equation}
\nonumber
\Gamma = \langle \widehat{R}_0, \widehat{R}_1, \dotsc, \widehat{R}_k \rangle.
\end{equation}
Moreover, if $i \geq 2$ and $\varphi \in \widehat{R}_i = R_{i-1}$, then $F_0\varphi = F_0$ and hence $\sigma_{F_0}\varphi = \varphi\sigma_{F_0}$; therefore, $\widehat{R}_0\widehat{R}_i = \widehat{R}_i\widehat{R}_0$ (with commutation occurring even at the level of elements).

We need to verify the intersection property \eqref{intprop} for $\Gamma$, that is $\Gamma_{I} \cap \Gamma_{J} = \Gamma_{I\cap J}$ for $I, J\subset\{0, \dotsc k\}$. If $0\notin I, J$ then $\Gamma_I$ and $\Gamma_J$ both lie in $\Gamma(\K)$, so we can directly appeal to the intersection property of $\Gamma(\K)$. If $0\in I,J$, write 
$I = \{0\} \cup I'$ and $J=\{0\}\cup J'$ with $I',J'\subseteq \{1,\ldots,k\}$. Then
\begin{align*}
\Gamma_{I}\cap \Gamma_{J} & 
= ((\bigotimes_{\varphi \in \Gamma_{I'}} \langle \sigma_{F_0\varphi} \rangle) \ltimes \Gamma_{I'}) 
\;\cap\; ((\bigotimes_{\varphi\in \Gamma_{J'}} \langle \sigma_{F_0\varphi} \rangle) \ltimes \Gamma_{J'})\\
&= (\bigotimes_{\varphi\in\Gamma_{I'}\cap \Gamma_{J'}}  \langle\sigma_{F_0\varphi}\rangle) \ltimes ( \Gamma_{I'} \cap \Gamma_{J'})
\;=\; (\bigotimes_{\varphi\in\Gamma_{I'\cap J'}} \langle \sigma_{F_0\varphi} \rangle) \ltimes \Gamma_{I'\cap J'}
\;=\; \Gamma_{I\cap J},
\end{align*}
by the semi-direct product structure of the groups and the intersection property of $\Gamma(\K)$. This settles the intersection property in the case  $0\in I, J$.  If only one set, $I$ (say), contains~$0$, and again $I=\{0\}\cup I'$, then 
\[ \Gamma_I\cap \Gamma_J 
\;=\; ((\bigotimes_{\varphi\in \Gamma_{I'}} \langle \sigma_{F_0\varphi} \rangle) \ltimes \Gamma_{I'}) \,\cap\, \Gamma_J
\;=\; \Gamma_{I'}\cap \Gamma_J = \Gamma_{I'\cap J}
\;=\; \Gamma_{I\cap J}, \]
again by the semi-direct product structure and the intersection property of $\Gamma(\K)$.

Thus the group $\Gamma = C_n\wr\Gamma(\K)$ is a flag-transitive subgroup of the automorphism group of a regular incidence complex of rank $k+1$. In fact, by \cite[Thm. 3]{kom2} this complex is necessarily isomorphic to $n^\K$. In fact, the automorphism group $\Gamma(n^{\mathcal{K}})$ of $n^\K$ has a flag-transitive subgroup isomorphic to $C_n\wr \Gamma(\K)$ (see Lemma~\ref{transsub}), and this subgroup acts on $n^\K$ in just the same manner as $\Gamma$ does on the regular $(k+1)$-complex associated with $\Gamma$.  

Thus $n^{\mathcal{K}}$ can be constructed from a small flag-transitive subgroup, $C_n\wr \Gamma(\K)$, of the full automorphism group $\Gamma(n^{\mathcal{K}}) = S_n\wr\Gamma(\K)$ by a twisting operation. 

In the above the cyclic group $C_n$ can be replaced by any transitive subgroup $U$ of $S_n$, and then $n^\K$ can be recovered from the flag-transitive subgroup $U\wr \Gamma(\K)$ of $S_n\wr\Gamma(\K)$. The choice of $C_n$ represents the smallest possible case. 

\section{Generalized power complexes and twisting}
\label{twist}

The results in this section are inspired by the twisting construction for the regular polytopes $\mathcal{L}^{\mathcal{K}, \mathcal{G}}$ described in \cite[Ch.~8B]{arp}.  This construction proceeds from a suitable Coxeter group $W$ on which $\Gamma(\K)$ acts as a group of group automorphisms, and then extends $W$ by $\Gamma(\K)$ to find $\mathcal{L}^{\mathcal{K},\mathcal{G}}$.  The Coxeter diagram for $W$ depends on the polytopes $\mathcal{K}$ and $\mathcal{L}$ and contains $\mathcal{G}$ as an induced subdiagram.  In particular, $\Gamma(\mathcal{L}^{\mathcal{K}, \mathcal{G}}) \cong W\ltimes \Gamma(\K)$.

In the context of arbitrary regular incidence complexes, the applicability of a similar technique is severely constrained by the lack of readily available classes of groups $W$ on which the whole group of a regular complex $\mathcal{K}$ can act in a suitable way.  Here we limit ourselves to two special cases for regular complexes $\mathcal{K}$ and~$\mathcal{L}$:
\begin{itemize}
\item $\mathcal{L}$ is a universal regular polytope of rank $d\geq 1$ (and then $W$ is a Coxeter group),
\item $\mathcal{L}$ has rank 1 (and $W$ is the direct product of cyclic groups).
\end{itemize}
In these two cases the twisting operation carries over and gives regular complexes which we denote again by $\mathcal{L}^{\mathcal{K}, \mathcal{G}}$. The second case was already investigated in Section~\ref{nktwist}; in fact if $\mathcal{L}$ is an $n$-edge $\{\}_n$ (and $\mathcal{G}$ is the trivial diagram), then $\mathcal{L}^{\mathcal{K}, \mathcal{G}} \cong n^{\mathcal{K}}$. In this section we study the first case.

Throughout this section, $\mathcal{L}$ is a universal regular $d$-polytope of type $\{q_1, \dotsc, q_{d-1}\}$ (see \cite[Ch. 3D]{arp}). Here ``universal'' means that $\mathcal{L}$ covers every regular $d$-polytope of the same Schl\"afli type. The automorphism group $\Gamma(\mathcal{L})$ of $\mathcal{L}$ is the Coxeter group with a string diagram on $d$ nodes, in which the branches are labeled $q_1, \dotsc, q_{d-1}$. We let $\Gamma(\mathcal{L}) = \langle \rho_0, \dotsc, \rho_{d-1}\rangle$, where $\rho_0, \dotsc, \rho_{d-1}$ are the distinguished generators.  

Now let $\mathcal{K}$ be a finite regular complex of rank $k\geq 1$ with automorphism group $\Gamma(\K) =\langle R_0,\dotsc,R_{k-1}\rangle$ (since $k \geq 1$ we can ignore mentioning $R_{-1}$ and $R_k$).  Let $\mathcal{G}$ be a Coxeter diagram in which the nodes are indexed by the vertices of $\mathcal{K}$; that is, the node set $V(\mathcal{G})$ of $\mathcal{G}$ is the vertex set $V(\mathcal{K})$ of $\mathcal{K}$. Now suppose $\Gamma(\K)$ acts node-transitively on $\mathcal{G}$ as a group of diagram symmetries, the vertex stabilizer-subgroup $\langle R_1, \dotsc, R_{k-1}\rangle$ of $\Gamma(\K)$ fixes the node~$F_0$~(say) of $\mathcal{G}$ corresponding to the base vertex of $\mathcal{K}$ (it may fix more than one node), and the action of $\Gamma(\K)$ on $\mathcal{G}$ respects the following intersection property for subsets of nodes of $\mathcal{G}$ defined in terms of the generating subgroups $R_0, \dotsc, R_{k-1}$ of $\Gamma(\K)$. The latter means that if $V(\mathcal{G}, I)$ denotes the set of images of $F_0$ under the subgroup $\langle R_i|\; i\in I\rangle$ of $\Gamma(\K)$ for $I\subseteq \{0, \dotsc, k-1\}$, then 
\begin{equation}
\label{nodeint}
V(\mathcal{G}, I)\cap V(\mathcal{G}, J) = V(\mathcal{G},I\cap J)\quad(I, J\subseteq \{0, \dotsc, k-1\}).
\end{equation}
In applications, the action of $\Gamma(\K)$ on the nodes of $\mathcal{G}$ is just the standard action of $\Gamma(\K)$ on the vertices of $\mathcal{K}$. In this case the intersection condition \eqref{nodeint} is satisfied if $\mathcal{K}$ is vertex-describable. We will make this assumption from now on.

Given $\mathcal{L}$ and $\mathcal{K}$ as above we now merge the string Coxeter diagram of $\mathcal{L}$ with the Coxeter diagram $\mathcal{G}$ to obtain a larger diagram $\mathcal{D}$, which also admits an action of $\Gamma(\K)$ as a group of diagram symmetries.  More precisely, we extend the diagram for $\mathcal{L}$ by the diagram $\mathcal{G}$ to a diagram $\mathcal{D}$, by identifying the node $F_0$ of $\mathcal{G}$ with the  node $d-1$  of the diagram of $\mathcal{L}$ and adding, for each $G\in V(\mathcal{G})$, a  node labeled $G$ and a branch marked $q_{d-1}$ between nodes ${d-2}$ and $G$ (see Figure~\ref{twistdia}). In addition, any branches and labels from $\mathcal{G}$ are included in $\mathcal{D}$ (Figure~\ref{twistdia} suppresses any such branches).  Then the node set of $\mathcal{D}$ is
\[V(\mathcal{D}) :=V(\mathcal{G}) \cup \{0,\dotsc, d-2\},\]
where here the node $d-1$ of $\mathcal{D}$ is viewed as lying in $V(\mathcal{G})$.

\begin{figure}[ht!]
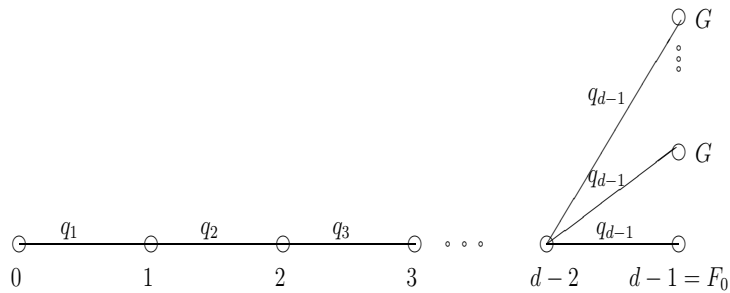

\centering
\resizebox{260pt}{110pt}{\input fig32.latex}
\caption{The Coxeter diagram $\mathcal{D}$.}
\label{twistdia}
\end{figure}

Throughout the remainder of this discussion we focus on the case when $\mathcal{G}$ is the {\em trivial\/} diagram, with nodes the vertices of $\mathcal{K}$ and without any branches. (The method works in more general contexts but we restrict ourselves to this case.) In other words, the Coxeter group with diagram $\mathcal{G}$ is $C_2^v$, where again $v$ is the number of vertices of $\mathcal{K}$. We also simplify notation and write $\mathcal{L}^\K$ in place of $\mathcal{L}^{\K,\mathcal{G}}$ (when $\mathcal{G}$ is the trivial diagram).

The group $\Gamma(\K)$ acts on $\mathcal{D}$ as a group of diagram symmetries and therefore also on the Coxeter group
\[ W  := \langle \sigma_H |\; H\in V(\mathcal{D})\rangle\]
defined by $\mathcal{D}$ as a group of group automorphisms permuting the generators and fixing $\sigma_0, \dotsc, \sigma_{d-2}$. Note that $W$ contains $\Gamma(\mathcal{L})$ as a subgroup.

We generate the desired $(d+k)$-incidence complex $\mathcal{L}^{\mathcal{K}}$ from the subgroup 
\[\Gamma := \langle \widehat{R}_0,\widehat{R}_1,\dotsc, \widehat{R}_{d+k-1}\rangle\]
of $W\ltimes \Gamma(\K)$, where
\begin{equation}
\label{gengamma}
\widehat{R}_i := 
\begin{cases}
\langle \rho_i\rangle & \mbox{if } i = 0,\dotsc, d-1\\
R_{i-d} & \mbox{if } i = d, d+1, \dotsc, d+k-1.
\end{cases}
\end{equation}

\begin{theorem}
\label{twistthm}
The group $\Gamma$ is a flag-transitive subgroup of the automorphism group of a regular incidence complex of rank $d+k$, denoted $\mathcal{L}^{\mathcal{K}}$. Moreover, $\Gamma = W\ltimes \Gamma(\K)$.
\end{theorem}

\begin{proof}
Since $\Gamma(\K) = \langle R_0,\dotsc, R_{d-1}\rangle$ acts transitively on the vertices of $\mathcal{K}$, each generator $\sigma_H$ of $W$ lies in $\Gamma$.  In fact, if  $F\in V(\mathcal{K})$ there is a $\varphi\in \Gamma(\K)$ such that $F=F_0\varphi$, and then 
\[\sigma_F = \varphi^{-1}\sigma_{F_0}\varphi = \varphi^{-1}\rho_{d-1}\varphi .\] 
It then follows that $\Gamma = W\ltimes \Gamma(\K)$.

Since $\Gamma(\K)$ fixes the nodes $0, \dotsc, d-2$ of $\mathcal{D}$, it centralizes the generators $\rho_0, \dotsc, \rho_{d-2}$ of $\Gamma$ and therefore $\langle \rho_i\rangle R_j = R_j\langle \rho_i\rangle$ for $0 \leq i \leq d-2$, $0 \leq j\leq k-1$.  Thus $\widehat{R}_i\widehat{R}_j=\widehat{R}_j\widehat{R}_i$ if $|i-j|\geq 2$.

To see that $\Gamma$ satisfies the intersection property \eqref{intprop}, let $I, J \subseteq \{0, \dotsc, d+k-1\}$. We need to show that  $\Gamma_{I\cap J}= \Gamma_I \cap \Gamma_J$. Write $I=I_1 \cup I_2$ with $I_1\subseteq \{0,\dotsc, d-1\}$ and $I_2\subseteq \{d, d+1, \dotsc, d+k-1\}$.  Then
\begin{align}
\label{twistintprop}
\Gamma_I &\,=\, \langle \widehat{R}_i|\;i\in I\rangle  
\,=\,  \langle \rho_i, R_{j-d} |\; i\in I_1\mbox{, }j\in I_2\rangle\nonumber\\[.01in]
&\,=\,
\begin{cases}
\langle \rho_i|\; i\in I_1\rangle\, \langle R_{j-d}|\,j\in I_2\rangle &\mbox{if }d-1\not\in  I_1\\[.01in]
\langle \rho_i, \sigma_H|\; i\in I_1 , H\in V(\mathcal{G}, I_2 -d)\rangle\,\langle R_{j-d}|\; j\in I_2\rangle &\mbox{if }d-1\in I_1,
\end{cases}
\end{align}
where here $I_{2}-d:=\{j-d\mid j\in I_2\}$. The products in (\ref{twistintprop}) are semi-direct products, and when $d-1\not\in I_1$ the product is even direct.  In any case, the first factor is a subgroup of $W$ and the second factor is a subgroup of  $\Gamma(\K)$. For $J$ we can similarly define $J_1$ and $J_2$ so that $J=J_1\cup J_2$. Clearly, 
\[I \cap J = (I_1\cup I_2)\cap (J_1\cup J_2) = (I_1\cap J_1)\cup(I_2\cap J_2).\]
Now the intersection property of $W$ and $\Gamma(\K)$ gives  
\[  \Gamma_{I_1\cap J_1} = \Gamma_{I_1}\cap \Gamma_{J_1},\quad
\Gamma_{I_2\cap J_2} = \Gamma_{I_2}\cap \Gamma_{J_2}, \] 
and the semi-direct product structure implies that $\Gamma_{I\cap J} = \Gamma_I \cap \Gamma_J$. Hence $\Gamma$ satisfies the intersection property.

Thus $\Gamma$ has all properties required of a flag-transitive subgroup of a regular $(d+k)$-complex. Then by \cite{kom2} we know that indeed there is such a complex admitting $\Gamma$ as flag-transitive group, and we denote it by $\mathcal{L}^\K$.
\end{proof}

Note that if $\{F_{-1},F_0,\dotsc, F_k\}$ is the base flag of $\mathcal{K}$, then for $d\leq j\leq d+k-1$ we have
\begin{align}
\label{8B7}
\langle \widehat{R}_0,\dotsc, \widehat{R}_j\rangle \,=\, 
\langle \rho_i, \sigma_H|\; 0\leq i \leq d-1, H \in V(\mathcal{G},\{0,\dotsc, j-d\})\rangle 
\,\langle R_0,\dotsc, R_{j-d}\rangle ,
\end{align}
which again a semi-direct product; here
\[V(\mathcal{G},\{0,\dotsc, j-d\}) = \{H\in V(\mathcal{K})|\; H\leq F_{j-d+1}\}.\]

The following theorem describes the structure of the faces and co-faces of the regular complex $\mathcal{L}^\mathcal{K}$ derived from $\Gamma$.

\begin{theorem}
\label{subcomplex}
Let $\mathcal{L}$ be a universal regular $d$-polytope $\{q_1, \dotsc, q_{d-1}\}$ and $\mathcal{K}$ a finite regular $k$-incidence complex.  For $1\leq i\leq k$ let $\mathcal{K}_i$ denote the isomorphism type of the $i$-faces of $\mathcal{K}$. Then the $(d+i)$-faces of $\mathcal{L}^{\mathcal{K}}$ are isomorphic to $\mathcal{L}^{\mathcal{K}_i}$.
Additionally, if $0\leq i\leq d-2$  then the co-faces at $i$-faces of  $\mathcal{L}^{\mathcal{K}}$  are isomorphic to $\mathcal{L}_i^{\mathcal{K}}$, where $\mathcal{L}_i$ is the universal regular polytope $\{q_{i+2}, \dotsc, q_{d-1}\}$, the co-face at an $i$-face of $\mathcal{L}$.
\end{theorem}

\begin{proof}
Let $\mathcal{G}_i$ denote the induced subdiagram of $\mathcal{G}$ on the subset $V(\mathcal{G},\{0, \dotsc, i-1\})$ of~$V(\mathcal{G})$. Then $\mathcal{G}_i$ is trivial since $\mathcal{G}$ is trivial. For the first part of the theorem, notice that $V(\mathcal{G},\{0,\dotsc, i-1\})=V(\mathcal{G}_i)$. Then apply (\ref{8B7}) with $j=d+i-1$ to obtain 
\begin{align}
\label{facestr}
\langle \widehat{R}_0,\dotsc, \widehat{R}_{d+i-1}\rangle &\,=\, \langle \rho_k, \sigma_H|\; 0\leq k\leq d-1, H\in V(\mathcal{G}_i)\rangle\, \langle R_0,\dotsc, R_{i-1}\rangle \nonumber\\
&\,=\, \langle \rho_k, \sigma_H|\; 0\leq k\leq d-1, H\in V(\mathcal{G}_i)\rangle\rtimes \langle R_0,\dotsc, R_{i-1}\rangle.\nonumber
\end{align}
Because of the specific semi-direct product structure, this latter group is precisely the analogue of $\Gamma$ obtained when $\mathcal{K}$ and $\mathcal{G}$ are replaced by $\mathcal{K}_i$ and $\mathcal{G}_i$, respectively. It follows that the $(d+i)$-faces of $\mathcal{L}^{\mathcal{K}}=\mathcal{L}^{\mathcal{K}, \mathcal{G}}$ are isomorphic to $\mathcal{L}^{\mathcal{K}_i}=\mathcal{L}^{\mathcal{K}_i, \mathcal{G}_i}$. For the second part, the construction using the subdiagram of $\mathcal{D}$ induced on $V(\mathcal{G}) \cup \{i+1, \dotsc, d-2\}$ is the same as the construction using  $\mathcal{L}_i$, $\mathcal{K}$ and $\mathcal{G}$.
\end{proof}

Theorem~\ref{subcomplex} says that the subgroups associated with the facet and vertex-figure of a regular complex $\mathcal{L}^{\mathcal{K}}$ are what we would expect.  For the facet this subgroup is $\langle \widehat{R}_0,\dotsc, \widehat{R}_{d+k-2}\rangle$, and for the vertex-figure it is $\langle \widehat{R}_1,\dotsc, \widehat{R}_{d+k-1}\rangle$. Note that these subgroups may not be the full automorphism groups of the facet and vertex-figure, respectively.

Note that the power complex $\mathcal{L}^{\mathcal{K}}$ where both $\mathcal{L}$ and $\K$ are simplices is itself a universal regular polytope with a Coxeter group as automorphism group. In fact, 
\begin{equation}
\label{simsim}
\{3^{d-1}\}^{\{3^{k-1}\}} = \{3^{d-1}, 4, 3^{k-1}\}
\end{equation}
(see \cite[Cor. 8B10]{arp}). In particular,
\[ \{3\}^{\{3\}} = \{3,4,3\}, \]
the $24$-cell.
\medskip

We remark that in the construction of the complexes $\mathcal{L}^{\mathcal{K}}$ the automorphism group $\Gamma(\K)$ of $\mathcal{K}$ can be replaced by any flag-transitive subgroup, $\Lambda$ (say), of $\Gamma(\K)$, in the sense that the resulting regular complex obtained from the new group $W\ltimes \Lambda$ (with the same $W$) is isomorphic to $\mathcal{L}^{\mathcal{K}}$. The group $W\ltimes \Lambda$ is a flag-transitive subgroup of $W\ltimes \Gamma(\K)$ and hence of $\Gamma(\mathcal{L}^{\mathcal{K}})$. Its distinguished generator subgroups are given by \[ \langle\rho_0\rangle, \dotsc, \langle\rho_{d-1}\rangle,R'_0, \dotsc, R'_{k-1},\] 
where $R'_0, \dotsc, R'_{k-1}$ are the distinguished generator subgroups of $\Lambda$.

The construction of generalized power complexes $\mathcal{L}^{\mathcal{K}}$ also behaves nicely with respect to taking skeletons.

\begin{theorem}
\label{twistskel}
Let $\mathcal{L}$ be a universal regular $d$-polytope, let $\mathcal{K}$ be a finite regular $k$-incidence complex, and let $j \leq k-1$. Then 
\[\mathrm{skel}_{d+j}(\mathcal{L}^{\mathcal{K}}) = \mathcal{L}^{\mathrm{skel}_{j}(\mathcal{K})}.\]
\end{theorem}

\begin{proof}
As before let $\Gamma := W\ltimes \Gamma(\K) = \langle\widehat{R}_0,\dotsc, \widehat{R}_{d+k-1}\rangle$, where the distinguished generators are as in \eqref{gengamma}. Recall that $\Gamma$ is flag-transitive subgroup of $\Gamma(\mathcal{L}^{\mathcal{K}})$.

First note that $\mathrm{skel}_j(\mathcal{K})$ is a regular $(j+1)$-complex on which $\Gamma(\K)$ acts flag-transitively as a group of automorphisms. Relative to $\mathrm{skel}_j(\mathcal{K})$, the distinguished generator subgroups of $\Gamma(\K)$ are given by $R_0, \dotsc, R_{j-1}, R$, where $R=\langle R_j, \dotsc, R_{k-1}\rangle$.

Now bear in mind the remark preceding the theorem. Then it is clear that the regular complex $\mathcal{L}^{\mathrm{skel}_j(\mathcal{K})}$ can be constructed from the flag-transitive subgroup  $\Gamma(\K)$ of $\Gamma(\mathrm{skel}_j(\mathcal{K}))$, rather than from $\Gamma(\mathrm{skel}_j(\mathcal{K}))$ itself. The resulting flag-transitive subgroup of $\Gamma(\mathcal{L}^{\mathrm{skel}_j(\mathcal{K})})$ then is again $\Gamma$, with distinguished generators relative to $\mathcal{L}^{\mathrm{skel}_j(\mathcal{K})}$ given by $\widehat{R}_0,\dotsc,\widehat{R}_{d-1}, R_0,\dotsc, R_{j-1}, R$.

On the other hand, the $(d+j)$-skeleton of $\mathcal{L}^{\mathcal{K}}$ admits the subgroup $\langle \widehat{R}_0,\dotsc, \widehat{R}_{d+j-1}, \widehat{R}\rangle$, with $\widehat{R} := \langle \widehat{R}_{d+j},\dotsc, \widehat{R}_{d+k-1}\rangle$, as a flag-transitive subgroup.  But $\widehat{R}_i = R_{i-d}$ for $i\geq d$, so $\widehat{R} = R$ and 
\[\langle \widehat{R}_0,\dotsc, \widehat{R}_{d+j-1}, \widehat{R}\rangle 
= \langle \widehat{R}_0,\dotsc, \widehat{R}_{d-1}, R_0,\dotsc, R_{j-1}, R\rangle 
= \Gamma.\]  
Hence, since the reconstruction of regular complexes from flag-transitive subgroups is unique, we must have 
$\mathrm{skel}_{d+j}(\mathcal{L}^{\mathcal{K}}) = \mathcal{L}^{\mathrm{skel}_j(\mathcal{K})}$.
\end{proof}

As an illustration of the previous theorem observe that
\begin{equation}
\{3^{d-1}\}^{\{\}_{k}} = \mathrm{skel}_{d}(\{3^{d-1}, 4, 3^{k-2}\}) .
\end{equation}
In fact, the $k$-edge $\{\}_k$ is the 0-skeleton of the $(k-1)$-simplex $\{3^{k-2}\}$, and by Theorem~\ref{twistskel} and \eqref{simsim}, we have 
\[\{3^{d-1}\}^{\{\}_k} = \{3^{d-1}\}^{\mathrm{skel}_0(\{3^{k-2}\})} = \mathrm{skel}_{d}(\{3^{d-1}\}^{\{3^{k-2}\}}) = \mathrm{skel}_{d}(\{3^{d-1}, 4, 3^{k-2}\}).\]
Thus the $(d+1)$-complex $\{3^{d-1}\}^{\{\}_k}$ is the $d$-skeleton of the regular $(d+k-1)$-polytope $\{3^{d-1},4,3^{k-2}\}$. In particular, 
\[ \{3\}^{\{\}_{3}} = \mathrm{skel}_{2}(\{3,4,3\}.\] 

The 3-complex $\{3\}^{\{\}_{3}}$ also occurs as a facet of the infinite regular $4$-complex $\{3\}^{\mathrm{PG}(2,2)}$. The diagram of the underlying Coxeter group $W$ is a star with seven unmarked branches, so both $W$ and $\Gamma=W\rtimes \mathrm{PGL}(3,2)$ are infinite groups. The latter is a flag-transitive subgroup of the automorphism group of $\{3\}^{\mathrm{PG}(2,2)}$. The vertex-figure of $\{3\}^{\mathrm{PG}(2,2)}$ is isomorphic to the $3$-complex $2^{\mathrm{PG}(2,2)}$ described earlier.
\medskip

The construction of $\mathcal{L}^{\mathcal{K},\mathcal{G}}$ described above extends more generally to regular complexes $\mathcal{K}$ and $\mathcal{L}$ for which a suitable group $W$, not necessarily a Coxeter group, can be found. To briefly discuss the limitations of this approach let again $\mathcal{L}$ and $\mathcal{K}$ be regular, of ranks $d$ and $k$ respectively, and let $\Gamma(\mathcal{L}) = \langle L_0,\dotsc, L_{d-1}\rangle$ and $\Gamma(\K) = \langle R_0,\dotsc, R_{k-1}\rangle$.  Unlike in the above construction of $\mathcal{L}^{\mathcal{K},\mathcal{G}}$ we are not assuming that $\mathcal{L}$ is a polytope.  Now suppose there exists a group $W$ generated by a distinguished family of subgroups, such that $\Gamma(\K)$ acts suitably on $W$ as a group of group automorphisms permuting the subgroups in this distinguished family. More precisely, suppose that $\Gamma(\mathcal{L})$ is a subgroup of this group $W$, and that the distinguished family of generating subgroups of $W$ (the analogues of the $\rho_{i}$'s and $\sigma_H$'s) is given by the generating subgroups $L_0,\dotsc, L_{d-1}$ of $\Gamma(\mathcal{L})$ and by subgroups $S_H$, $H\in V(\mathcal{K})$. Further assume that the subgroup $S_{F_0}$ associated with the base vertex $F_0$ of $\mathcal{K}$ coincides with the subgroup $L_{d-1}$ of $\Gamma(\mathcal{L})$, and that the action of $\Gamma(\K)$ on $W$ resembles the one used in the construction of $\mathcal{L}^{\mathcal{K},\mathcal{G}}$.  In other words, the elements $\varphi\in\Gamma(\K)$ leave each subgroup $L_i$ with $i\leq d-2$ invariant and take a subgroup $S_H$ to the subgroup $S_{(H)\varphi}$ for $H\in V(\mathcal{K})$, and the vertex stabilizer subgroup $\langle R_1, \dotsc, R_{k-1}\rangle$ of $\Gamma(\K)$ leaves the subgroup $S_{F_0}$ of $W$ invariant. Then this setup already allows us to construct the semi-direct product 
\[\Gamma:= W \ltimes \Gamma(\K) = \langle  \widehat{R}_0,\dotsc, \widehat{R}_{d+k-1}\rangle,\]
where
\begin{equation*}
\widehat{R}_i := 
\begin{cases}
L_i & \mbox{if }i = 0,\dotsc, d-1\\
R_{i-d} &\mbox{if } i = d, d+1, \dotsc, d+k-1.
\end{cases}
\end{equation*}
By our assumption on the action of the vertex stabilizer subgroup $\langle R_1,\dotsc, R_{k-1}\rangle$ on $W$ we have the desired commutativity property $\widehat{R}_i\widehat{R}_j = \widehat{R}_j\widehat{R}_i$ for $|i-j \geq 2$.  However, the crucial condition is the intersection property for $\Gamma$, and this can usually only be guaranteed if $W$ already has the intersection property with respect to its distinguished family of generator subgroups.  This severely restricts the possibilities and explains why we have focused on the cases outlined at  the beginning of this section. We should add that in the case where $\mathcal{L}$ is a polytope and the distinguished generator subgroups of $W$ are themselves generated by involutions, the intersection property of $W$ is satisfied if and only if $W$ is a C-group (see \cite[Ch. 2E]{arp}).
\smallskip

{\em Acknowledgement.}
We are very grateful to the anonymous referees for a number of helpful comments that have improved our article.

\bibliographystyle{amsplain}

\end{document}

%% file: fig24.latex
\setlength{\unitlength}{3947sp}%
\begingroup\makeatletter\ifx\SetFigFont\undefined%
\gdef\SetFigFont#1#2#3#4#5{%
  \reset@font\fontsize{#1}{#2pt}%
  \fontfamily{#3}\fontseries{#4}\fontshape{#5}%
  \selectfont}%
\fi\endgroup%
\begin{picture}(7830,2382)(1111,-7873)
\put(6601,-7786){\makebox(0,0)[lb]{\smash{{\SetFigFont{14}{16.8}{\familydefault}{\mddefault}{\updefault}{\color[rgb]{0,0,0}$\mathbf{1}^n_i$}%
}}}}
\thinlines
{\color[rgb]{0,0,0}\put(1126,-7561){\line( 0, 1){1800}}
}%
{\color[rgb]{0,0,0}\put(1126,-7561){\line( 3, 2){2751.923}}
}%
{\color[rgb]{0,0,0}\put(1126,-5761){\line( 3,-2){2648.077}}
}%
{\color[rgb]{0,0,0}\put(3901,-5761){\line( 6,-5){2080.328}}
}%
{\color[rgb]{0,0,0}\put(1126,-5761){\line( 5,-6){1530.738}}
}%
{\color[rgb]{0,0,0}\put(5326,-7561){\line(-4, 5){1434.146}}
}%
{\color[rgb]{0,0,0}\put(7576,-7561){\line(-1, 2){885}}
}%
{\color[rgb]{0,0,0}\put(3901,-5761){\line( 3,-2){2700}}
}%
{\color[rgb]{0,0,0}\put(6751,-5761){\line( 6,-5){2205.738}}
}%
{\color[rgb]{0,0,0}\put(6751,-5761){\line( 4,-5){1434.146}}
}%
\put(1126,-7786){\makebox(0,0)[lb]{\smash{{\SetFigFont{14}{16.8}{\familydefault}{\mddefault}{\updefault}{\color[rgb]{0,0,0}$\mathbf{1}$}%
}}}}
\put(2701,-7786){\makebox(0,0)[lb]{\smash{{\SetFigFont{14}{16.8}{\familydefault}{\mddefault}{\updefault}{\color[rgb]{0,0,0}$\mathbf{1}^j_1$}%
}}}}
\put(1126,-5686){\makebox(0,0)[lb]{\smash{{\SetFigFont{14}{16.8}{\familydefault}{\mddefault}{\updefault}{\color[rgb]{0,0,0}$G_1(\mathbf{1})$}%
}}}}
\put(5326,-7786){\makebox(0,0)[lb]{\smash{{\SetFigFont{14}{16.8}{\familydefault}{\mddefault}{\updefault}{\color[rgb]{0,0,0}$\mathbf{1}^2_i$}%
}}}}
\put(3751,-7786){\makebox(0,0)[lb]{\smash{{\SetFigFont{14}{16.8}{\familydefault}{\mddefault}{\updefault}{\color[rgb]{0,0,0}$\mathbf{1}^n_1$}%
}}}}
\put(3826,-5686){\makebox(0,0)[lb]{\smash{{\SetFigFont{14}{16.8}{\familydefault}{\mddefault}{\updefault}{\color[rgb]{0,0,0}$G_i(\mathbf{1})$}%
}}}}
\put(6676,-5686){\makebox(0,0)[lb]{\smash{{\SetFigFont{14}{16.8}{\familydefault}{\mddefault}{\updefault}{\color[rgb]{0,0,0}$G_v(\mathbf{1})$}%
}}}}
\put(8926,-7786){\makebox(0,0)[lb]{\smash{{\SetFigFont{14}{16.8}{\familydefault}{\mddefault}{\updefault}{\color[rgb]{0,0,0}$\mathbf{1}^n_v$}%
}}}}
\put(8176,-7786){\makebox(0,0)[lb]{\smash{{\SetFigFont{14}{16.8}{\familydefault}{\mddefault}{\updefault}{\color[rgb]{0,0,0}$\mathbf{1}^j_v$}%
}}}}
\put(7576,-7786){\makebox(0,0)[lb]{\smash{{\SetFigFont{14}{16.8}{\familydefault}{\mddefault}{\updefault}{\color[rgb]{0,0,0}$\mathbf{1}^2_v$}%
}}}}
\put(5926,-7786){\makebox(0,0)[lb]{\smash{{\SetFigFont{14}{16.8}{\familydefault}{\mddefault}{\updefault}{\color[rgb]{0,0,0}$\mathbf{1}^j_i$}%
}}}}
{\color[rgb]{0,0,0}\put(1126,-7561){\line( 3, 1){5602.500}}
}%
\end{picture}%

%% file: newfig28.latex
\setlength{\unitlength}{3947sp}%
\begingroup\makeatletter\ifx\SetFigFont\undefined%
\gdef\SetFigFont#1#2#3#4#5{%
  \reset@font\fontsize{#1}{#2pt}%
  \fontfamily{#3}\fontseries{#4}\fontshape{#5}%
  \selectfont}%
\fi\endgroup%
\begin{picture}(4583,3411)(2686,-5914)
\put(6751,-5236){\makebox(0,0)[lb]{\smash{{\SetFigFont{12}{14.4}{\familydefault}{\mddefault}{\updefault}{\color[rgb]{0,0,1}$5(\bf{1})$}%
}}}}
\put(4651,-4486){\makebox(0,0)[lb]{\smash{{\SetFigFont{12}{14.4}{\familydefault}{\mddefault}{\updefault}{\color[rgb]{1,0,0}$4(\bf{1})$}%
}}}}
\put(4951,-3361){\makebox(0,0)[lb]{\smash{{\SetFigFont{12}{14.4}{\familydefault}{\mddefault}{\updefault}$(12111111)$}}}}
\put(3601,-3136){\makebox(0,0)[lb]{\smash{{\SetFigFont{12}{14.4}{\familydefault}{\mddefault}{\updefault}$(22121111)$}}}}
\put(4951,-5386){\makebox(0,0)[lb]{\smash{{\SetFigFont{12}{14.4}{\familydefault}{\mddefault}{\updefault}$(11111211)$}}}}
\put(4351,-5611){\makebox(0,0)[lb]{\smash{{\SetFigFont{12}{14.4}{\familydefault}{\mddefault}{\updefault}{\color[rgb]{0,0,1}$3(\bf{1})$}%
}}}}
\put(3751,-5836){\makebox(0,0)[lb]{\smash{{\SetFigFont{12}{14.4}{\familydefault}{\mddefault}{\updefault}$(11211111)$}}}}
\put(6601,-5686){\makebox(0,0)[lb]{\smash{{\SetFigFont{12}{14.4}{\familydefault}{\mddefault}{\updefault}$(11112111)$}}}}
\put(6601,-2986){\makebox(0,0)[lb]{\smash{{\SetFigFont{12}{14.4}{\familydefault}{\mddefault}{\updefault}$(11111121)$}}}}
\put(6526,-3736){\makebox(0,0)[lb]{\smash{{\SetFigFont{12}{14.4}{\familydefault}{\mddefault}{\updefault}{\color[rgb]{0,0,1}$7(\bf{1})$}%
}}}}
\put(3001,-2686){\makebox(0,0)[lb]{\smash{{\SetFigFont{12}{14.4}{\familydefault}{\mddefault}{\updefault}$(22111111)$}}}}
\put(3451,-4561){\makebox(0,0)[lb]{\smash{{\SetFigFont{12}{14.4}{\familydefault}{\mddefault}{\updefault}$(21121111)$}}}}
\put(3226,-3886){\makebox(0,0)[lb]{\smash{{\SetFigFont{12}{14.4}{\familydefault}{\mddefault}{\updefault}$124(\bf{1})$}}}}
\put(5476,-5086){\makebox(0,0)[lb]{\smash{{\SetFigFont{12}{14.4}{\familydefault}{\mddefault}{\updefault}{\color[rgb]{0,0,1}$6(\bf{1})$}%
}}}}
\thinlines
{\color[rgb]{0,0,1}\put(5401,-4561){\line(-5,-1){2408.654}}
}%
{\color[rgb]{0,.56,0}\put(5401,-4561){\line( 0, 1){1050}}
}%
{\color[rgb]{1,0,0}\put(5401,-4561){\line(-1, 0){825}}
}%
{\color[rgb]{0,0,0}\put(3001,-5086){\line( 1, 6){372.973}}
}%
{\color[rgb]{0,0,0}\put(3676,-4261){\line(-5,-6){682.377}}
}%
{\color[rgb]{0,0,0}\put(3301,-2836){\line( 3,-1){2092.500}}
}%
{\color[rgb]{0,0,0}\put(4576,-4561){\line( 0, 1){1050}}
}%
{\color[rgb]{0,0,0}\put(5401,-3511){\line(-1, 0){825}}
}%
{\color[rgb]{0,0,0}\put(4501,-4561){\line(-3, 1){832.500}}
}%
{\color[rgb]{0,0,0}\put(3676,-4261){\line( 0, 1){975}}
}%
{\color[rgb]{0,0,0}\put(3301,-2836){\line( 1,-1){450}}
}%
{\color[rgb]{0,0,1}\put(5401,-4561){\line( 2,-1){1800}}
}%
{\color[rgb]{0,0,1}\put(7201,-3061){\line(-6,-5){1800}}
}%
{\color[rgb]{0,0,1}\put(6601,-4561){\line(-1, 0){1200}}
}%
{\color[rgb]{0,0,0}\put(4576,-3511){\line(-4, 1){829.412}}
}%
{\color[rgb]{0,0,1}\put(5401,-4561){\line(-5,-4){1207.317}}
}%
{\color[rgb]{0,0,1}\put(5401,-4561){\line( 0,-1){600}}
}%
\put(5251,-4861){\makebox(0,0)[lb]{\smash{{\SetFigFont{12}{14.4}{\familydefault}{\mddefault}{\updefault}{\color[rgb]{0,0,0}$\bf{1}$}%
}}}}
\put(5776,-4486){\makebox(0,0)[lb]{\smash{{\SetFigFont{12}{14.4}{\familydefault}{\mddefault}{\updefault}{\color[rgb]{0,0,1}$8(\bf{1})$}%
}}}}
\put(6601,-4786){\makebox(0,0)[lb]{\smash{{\SetFigFont{12}{14.4}{\familydefault}{\mddefault}{\updefault}$(11111112)$}}}}
\put(5401,-4111){\makebox(0,0)[lb]{\smash{{\SetFigFont{12}{14.4}{\familydefault}{\mddefault}{\updefault}{\color[rgb]{0,.56,0}$2(\bf{1})$}%
}}}}
\put(4276,-3811){\makebox(0,0)[lb]{\smash{{\SetFigFont{12}{14.4}{\familydefault}{\mddefault}{\updefault}$(12121111)$}}}}
\put(3526,-5161){\makebox(0,0)[lb]{\smash{{\SetFigFont{12}{14.4}{\familydefault}{\mddefault}{\updefault}{\color[rgb]{0,0,1}$1(\bf{1})$}%
}}}}
\put(2701,-5386){\makebox(0,0)[lb]{\smash{{\SetFigFont{12}{14.4}{\familydefault}{\mddefault}{\updefault}$(21111111)$}}}}
\put(3976,-4786){\makebox(0,0)[lb]{\smash{{\SetFigFont{12}{14.4}{\familydefault}{\mddefault}{\updefault}$(11121111)$}}}}
{\color[rgb]{0,0,0}\put(5401,-4561){\circle*{120}}
}%
{\color[rgb]{0,0,0}\put(4576,-3511){\circle*{120}}
}%
{\color[rgb]{0,0,0}\put(4576,-4561){\circle*{120}}
}%
{\color[rgb]{0,0,0}\put(5401,-3511){\circle*{120}}
}%
{\color[rgb]{0,0,0}\put(3001,-5086){\circle*{120}}
}%
{\color[rgb]{0,0,0}\put(3676,-4261){\circle*{120}}
}%
{\color[rgb]{0,0,0}\put(3751,-3286){\circle*{120}}
}%
{\color[rgb]{0,0,0}\put(7201,-5461){\circle*{120}}
}%
{\color[rgb]{0,0,0}\put(4201,-5536){\circle*{120}}
}%
{\color[rgb]{0,0,0}\put(7201,-3061){\circle*{120}}
}%
{\color[rgb]{0,0,0}\put(5401,-5161){\circle*{120}}
}%
{\color[rgb]{0,0,0}\put(3301,-2836){\circle*{120}}
}%
{\color[rgb]{0,0,0}\put(6601,-4561){\circle*{120}}
}%
\end{picture}%

%% file: fig210.latex
\setlength{\unitlength}{3947sp}%
\begingroup\makeatletter\ifx\SetFigFont\undefined%
\gdef\SetFigFont#1#2#3#4#5{%
  \reset@font\fontsize{#1}{#2pt}%
  \fontfamily{#3}\fontseries{#4}\fontshape{#5}%
  \selectfont}%
\fi\endgroup%
\begin{picture}(2401,2699)(4838,-2873)
\thinlines
{\color[rgb]{0,0,0}\put(7201,-1036){\line(-1, 0){900}}
}%
{\color[rgb]{0,0,0}\put(6001,-2836){\line( 0, 1){825}}
}%
{\color[rgb]{0,0,0}\put(6001,-736){\line(-3,-5){363.971}}
}%
{\color[rgb]{0,0,0}\put(6376,-1411){\line(-3,-5){363.971}}
}%
{\color[rgb]{0,0,0}\put(6001,-736){\line( 3,-5){397.059}}
}%
{\color[rgb]{0,0,0}\put(5626,-1336){\line( 3,-5){397.059}}
}%
{\color[rgb]{0,0,0}\put(4876,-1036){\line( 0, 1){825}}
}%
{\color[rgb]{0,0,0}\put(7201,-1036){\line( 0, 1){825}}
}%
{\color[rgb]{0,0,0}\put(4876,-1036){\line( 3,-5){311.029}}
}%
{\color[rgb]{0,0,0}\put(4876,-211){\line( 3,-5){311.029}}
}%
{\color[rgb]{0,0,0}\put(5701,-211){\line(-1, 0){825}}
}%
{\color[rgb]{0,0,0}\put(5701,-211){\line( 3,-5){311.029}}
}%
{\color[rgb]{0,0,0}\put(5701,-1036){\line( 0, 1){825}}
}%
{\color[rgb]{0,0,0}\put(6001,-736){\line(-1, 0){825}}
}%
{\color[rgb]{0,0,0}\put(5701,-1036){\line(-1, 0){825}}
}%
{\color[rgb]{0,0,0}\put(6301,-211){\line(-3,-5){311.029}}
}%
{\color[rgb]{0,0,0}\put(6301,-1036){\line( 0, 1){825}}
}%
{\color[rgb]{0,0,0}\put(7201,-211){\line(-1, 0){900}}
}%
{\color[rgb]{0,0,0}\put(6901,-736){\line(-1, 0){900}}
}%
{\color[rgb]{0,0,0}\put(7201,-211){\line(-3,-5){311.029}}
}%
{\color[rgb]{0,0,0}\put(7201,-1036){\line(-3,-5){311.029}}
}%
{\color[rgb]{1,0,0}\put(5701,-1036){\line( 3,-5){311.029}}
}%
{\color[rgb]{1,0,0}\put(6001,-1561){\line( 0, 1){825}}
}%
{\color[rgb]{1,0,0}\put(6301,-1036){\line(-3,-5){311.029}}
}%
{\color[rgb]{1,0,0}\put(6901,-1561){\line(-1, 0){900}}
}%
{\color[rgb]{1,0,0}\put(6001,-1561){\line( 3,-5){397.059}}
}%
{\color[rgb]{1,0,0}\put(6001,-1561){\line(-1, 0){825}}
}%
{\color[rgb]{0,0,0}\put(5176,-1561){\line( 0, 1){825}}
}%
{\color[rgb]{0,0,0}\put(6901,-1561){\line( 0, 1){825}}
}%
{\color[rgb]{0,0,0}\put(6376,-2236){\line( 0, 1){825}}
}%
{\color[rgb]{0,0,0}\put(5626,-2161){\line( 0, 1){825}}
}%
{\color[rgb]{1,0,0}\put(6001,-1561){\line(-3,-5){363.971}}
}%
{\color[rgb]{0,0,0}\put(5626,-2161){\line( 3,-5){397.059}}
}%
{\color[rgb]{0,0,0}\put(6369,-2232){\line(-3,-5){363.971}}
}%
{\color[rgb]{0,0,0}\put(6001,-736){\circle*{60}}
}%
{\color[rgb]{0,0,0}\put(6901,-1561){\circle*{60}}
}%
{\color[rgb]{0,0,0}\put(5176,-1561){\circle*{60}}
}%
{\color[rgb]{0,0,0}\put(5701,-1036){\circle*{60}}
}%
{\color[rgb]{0,0,0}\put(6301,-1036){\circle*{60}}
}%
{\color[rgb]{0,0,0}\put(5626,-2161){\circle*{60}}
}%
{\color[rgb]{0,0,0}\put(6376,-2236){\circle*{60}}
}%
{\color[rgb]{0,0,0}\put(6001,-2836){\circle*{60}}
}%
{\color[rgb]{0,0,0}\put(6001,-2011){\circle*{60}}
}%
{\color[rgb]{0,0,0}\put(6376,-1411){\circle*{60}}
}%
{\color[rgb]{0,0,0}\put(5626,-1336){\circle*{60}}
}%
{\color[rgb]{0,0,0}\put(4876,-1036){\circle*{60}}
}%
{\color[rgb]{0,0,0}\put(4876,-211){\circle*{60}}
}%
{\color[rgb]{0,0,0}\put(5176,-736){\circle*{60}}
}%
{\color[rgb]{0,0,0}\put(5701,-211){\circle*{60}}
}%
{\color[rgb]{0,0,0}\put(6301,-211){\circle*{60}}
}%
{\color[rgb]{0,0,0}\put(6901,-736){\circle*{60}}
}%
{\color[rgb]{0,0,0}\put(7201,-1036){\circle*{60}}
}%
{\color[rgb]{0,0,0}\put(7201,-211){\circle*{60}}
}%
{\color[rgb]{0,0,0}\put(6001,-1561){\circle*{60}}
}%
\end{picture}%

%% file: fig32.latex
\setlength{\unitlength}{3947sp}%
\begingroup\makeatletter\ifx\SetFigFont\undefined%
\gdef\SetFigFont#1#2#3#4#5{%
  \reset@font\fontsize{#1}{#2pt}%
  \fontfamily{#3}\fontseries{#4}\fontshape{#5}%
  \selectfont}%
\fi\endgroup%
\begin{picture}(6255,2112)(1561,-3439)
\put(6826,-2011){\makebox(0,0)[lb]{\smash{{\SetFigFont{12}{14.4}{\familydefault}{\mddefault}{\updefault}{\color[rgb]{0,0,0}$q_{d-1}$}%
}}}}
{\color[rgb]{0,0,0}\thinlines
\put(7651,-1786){\circle{30}}
}%
{\color[rgb]{0,0,0}\put(7651,-1636){\circle{30}}
}%
{\color[rgb]{0,0,0}\put(6451,-3061){\circle{120}}
}%
{\color[rgb]{0,0,0}\put(7651,-3061){\circle{120}}
}%
{\color[rgb]{0,0,0}\put(5551,-3061){\circle{30}}
}%
{\color[rgb]{0,0,0}\put(5701,-3061){\circle{30}}
}%
{\color[rgb]{0,0,0}\put(5851,-3061){\circle{30}}
}%
{\color[rgb]{0,0,0}\put(5251,-3061){\circle{120}}
}%
{\color[rgb]{0,0,0}\put(1651,-3061){\circle{120}}
}%
{\color[rgb]{0,0,0}\put(2851,-3061){\circle{120}}
}%
{\color[rgb]{0,0,0}\put(4051,-3061){\circle{120}}
}%
{\color[rgb]{0,0,0}\put(7651,-1411){\circle{120}}
}%
{\color[rgb]{0,0,0}\put(7651,-2386){\circle{120}}
}%
{\color[rgb]{0,0,0}\put(2851,-3061){\line( 1, 0){1200}}
}%
{\color[rgb]{0,0,0}\put(4051,-3061){\line( 1, 0){1200}}
}%
{\color[rgb]{0,0,0}\put(6451,-3061){\line( 1, 0){1200}}
}%
{\color[rgb]{0,0,0}\put(1651,-3061){\line( 1, 0){1200}}
}%
{\color[rgb]{0,0,0}\put(6451,-3061){\line( 5, 3){1180.147}}
}%
{\color[rgb]{0,0,0}\put(6451,-3061){\line( 3, 4){1224}}
}%
\put(2026,-2986){\makebox(0,0)[lb]{\smash{{\SetFigFont{12}{14.4}{\familydefault}{\mddefault}{\updefault}{\color[rgb]{0,0,0}$q_1$}%
}}}}
\put(3301,-2986){\makebox(0,0)[lb]{\smash{{\SetFigFont{12}{14.4}{\familydefault}{\mddefault}{\updefault}{\color[rgb]{0,0,0}$q_2$}%
}}}}
\put(4501,-2986){\makebox(0,0)[lb]{\smash{{\SetFigFont{12}{14.4}{\familydefault}{\mddefault}{\updefault}{\color[rgb]{0,0,0}$q_3$}%
}}}}
\put(2776,-3361){\makebox(0,0)[lb]{\smash{{\SetFigFont{12}{14.4}{\familydefault}{\mddefault}{\updefault}{\color[rgb]{0,0,0}$1$}%
}}}}
\put(3976,-3361){\makebox(0,0)[lb]{\smash{{\SetFigFont{12}{14.4}{\familydefault}{\mddefault}{\updefault}{\color[rgb]{0,0,0}$2$}%
}}}}
\put(5176,-3361){\makebox(0,0)[lb]{\smash{{\SetFigFont{12}{14.4}{\familydefault}{\mddefault}{\updefault}{\color[rgb]{0,0,0}$3$}%
}}}}
\put(6301,-3361){\makebox(0,0)[lb]{\smash{{\SetFigFont{12}{14.4}{\familydefault}{\mddefault}{\updefault}{\color[rgb]{0,0,0}$d-2$}%
}}}}
\put(7201,-3361){\makebox(0,0)[lb]{\smash{{\SetFigFont{12}{14.4}{\familydefault}{\mddefault}{\updefault}{\color[rgb]{0,0,0}$d-1=F_0$}%
}}}}
\put(1576,-3361){\makebox(0,0)[lb]{\smash{{\SetFigFont{12}{14.4}{\familydefault}{\mddefault}{\updefault}{\color[rgb]{0,0,0}$0$}%
}}}}
\put(6901,-2986){\makebox(0,0)[lb]{\smash{{\SetFigFont{12}{14.4}{\familydefault}{\mddefault}{\updefault}{\color[rgb]{0,0,0}$q_{d-1}$}%
}}}}
\put(7801,-1486){\makebox(0,0)[lb]{\smash{{\SetFigFont{12}{14.4}{\familydefault}{\mddefault}{\updefault}{\color[rgb]{0,0,0}$G$}%
}}}}
\put(7801,-2461){\makebox(0,0)[lb]{\smash{{\SetFigFont{12}{14.4}{\familydefault}{\mddefault}{\updefault}{\color[rgb]{0,0,0}$G$}%
}}}}
\put(6826,-2611){\makebox(0,0)[lb]{\smash{{\SetFigFont{12}{14.4}{\familydefault}{\mddefault}{\updefault}{\color[rgb]{0,0,0}$q_{d-1}$}%
}}}}
{\color[rgb]{0,0,0}\put(7651,-1711){\circle{30}}
}%
\end{picture}%